\documentclass[a4paper]{amsart}
\usepackage[all]{xy}
 \def\cal#1{\mathcal{#1}}

\def\AA{{\mathbb{A}}}
\def\PP{{\mathbb{P}}}
\def\RR{{\mathbb{R}}}

\def\NN{{\mathbb{N}}}
\def\ZZ{{\mathbb{Z}}}

\let \cedilla =\c

\renewcommand{\b}{{\frak{b}}}
\renewcommand{\a}{{\frak{a}}}
\renewcommand{\o}{{\mathcal O}}
\newcommand{\mld}{{\mathrm{mld}}}
\newcommand{\mldmj}{{\mathrm{mld_{MJ}}}}
\newcommand{\Hom}{\mathrm{Hom}}
\renewcommand{\j}{{\mathcal{J}}}
\newcommand{\cont}{{\mathrm{Cont}}}
\newcommand{\codim}{{\mathrm{codim}}}
\newcommand{\ord}{{\mathrm{ord}}}
\newcommand{\mult}{{\mathrm{mult}}}
\newcommand{\spec}{{\mathrm{Spec}}}
\newcommand{\hke}{{\widehat{k}_E}}
\newcommand{\hk}{{\widehat{k}}}
\newcommand{\I}{{\mathcal{I}}}
\newcommand{\val}{{\mathrm{val}}}
\newcommand{\jx}{{{\mathcal{J}}_X}}
\newcommand{\amj}{{a_{MJ}}}
\newcommand{\sing}{{\mathrm{Sing}}}
\newcommand{\emb}{{\mathrm{emb}}}
\renewcommand{\H}{{\mathcal{H}}}
\newcommand{\cha}{{\operatorname{char}}}

\newtheorem{thm}{Theorem}[section]
\newtheorem{cor}[thm]{Corollary}
\newtheorem{Corollary-Definition}[thm]{Corollary-Definition}
\newtheorem{prop}[thm]{Proposition}
\newtheorem{lem}[thm]{Lemma}

\newtheorem{conj}[thm]{Conjecture}

\theoremstyle{definition}
\newtheorem{defn}[thm]{Definition}

\newtheorem{rem}[thm]{Remark}
\newtheorem{say}[thm]{}
\newtheorem{claim}[thm]{Claim}

\newtheorem{Proposition-Definition}[thm]{Proposition-Definition}

\begin{document}

\title[]{Singularities  in arbitrary characteristic\\
via jet schemes}

\author{Shihoko Ishii and Ana J. Reguera}

\maketitle

\begin{abstract} This paper summarizes
recent results concerning 
 singularities with respect to the
Mather--Jacobian log discrepancies over an algebraically closed field of arbitrary characteristic.
The basic point is that the inversion of adjunction with respect to Mather--Jacobian discrepancies holds under arbitrary characteristic.
Using this fact, we will reduce several geometric properties of the singularities 
to jet scheme problems and try to avoid  discussions that are distinctive to characteristic 0.

\end{abstract}
\vskip.5truecm
\section{Introduction}
\noindent
Canonical  and log-canonical singularities play important roles  in
birational geometry over the base field of characteristic 0  
and  are recognized as ``good singularities," which we admit on a minimal
model.
These singularities can be well described by jet schemes when they are locally a complete
intersection.

On the other hand,  \cite{dd} and \cite{is}  independently introduced singularities that can be described by jet schemes in general.
These are based on Mather--Jacobian (MJ) discrepancies.
We define MJ-canonical (MJ-log-canonical) singularities
based on MJ discrepancy in a similar way to canonical (log-canonical ) singularities.
These MJ-canonical and MJ-log-canonical singularities have the following good properties:
rationality (MJ-canonical),
stability under small deformations (see \cite{ei}), and MJ-multiplier ideals (see \cite{EIM}).
These are all for a characteristic 0 case.

Now, for a positive characteristic case, one can also define
 canonical and log-canonical singularities as well as MJ-canonical and MJ-log-canonical
singularities.
However, canonical and log-canonical singularities of positive characteristic are difficult to treat, because the following results are  unavailable:
\begin{enumerate}
\item resolution of singularities;
\item Bertini's second theorem (generic smoothness); and
\item the Kodaira  vanishing theorem,
\end{enumerate}
all of which are used in  discussions of singularities over characteristic 0.

For example, even the seemingly simple statement that ``a general hyperplane section of a quasi-projective variety  has  at worst canonical singularities  if the variety has such singularities'' is not yet  proved for positive
characteristic case.
This statement has been proved in the characteristic 0 case by  resolution of the singularities and Bertini's second theorem.
If the variety is of dimension 3 over a base field of positive characteristic, then a resolution of the singularities exists
\cite{CP}, but  this problem is not
yet proved  (although  a result has been found under certain conditions \cite{hiro}).

On the contrary, the MJ-version of this problem has the possibility to be solved,
because MJ-singularities are well described by the jet schemes, which acts as an alternative to Bertini's
second theorem.
As  evidence, in Corollary \ref{hyper}, we prove this statement for three-dimensional MJ-canonical quasi-projective
variety,
thereby showing the possibility of MJ discrepancy for discussion of positive characteristics.

\vskip.5truecm
One of the aims of this paper is to summarize the results of MJ-singularities in the positive characteristic case and to clarify
those which has been proved and those which has not yet been proved.
Since  such a review 
does not presently exist elsewhere, we think that it would be useful for us to summarize them here.
 The basic theorems showing that MJ-singularities are well described by jet schemes
 are  Theorems \ref{formula} and  \ref{ia}.
 These theorems are proved in the same line as their proof   in the characteristic 
 0-case, by carefully avoiding the use of resolutions of the singularities.

The other aim of this paper is to show  that jet schemes are useful in the study of
singularities over positive characteristic base fields (Corollary \ref{hyper}).

 This paper is structured as follows:
 in Section 2,
 we give  characteristic-free discussions of MJ discrpancies.
In Section 3, we give preliminary information about jet schemes, define the codimension of
a cylinder, and describe minimal log MJ discrepancies using jet schemes in the 
arbitrary-characteristic case.
Note that these have been already established in the characteristic-0 case.
In Section 4, we show some properties of MJ-canonical and MJ-log-canonical singularities
of positive characteristic that have been proved and list some open problems for the 
positive-characteristic case.

\vskip.5truecm
{\bf Acknowledgement} The authors express their hearty thanks to Masayuki Hirokado for
providing them with information on varieties over the base field of positive characteristic.
We also thank the referees for constructible suggestions to improve the paper.

\vskip.5truecm
\section{Mather--Jacobian discrepancy}
\noindent
Throughout this paper, a variety refers to a reduced pure-dimensional scheme of
finite type over an algebraically closed field $k$ of arbitrary characteristic, unless
otherwise stated.

\begin{defn}
Let $X$ be a variety of dimension $d$.
Note that the projection
$$
\pi \colon \PP_X(\wedge^n \Omega_X) \to X
$$
is an isomorphism over the smooth locus $X_{reg} \subseteq X$.
In particular, we have a section $\sigma \colon X_{reg} \to \PP_X(\wedge^n \Omega_X)$.

The closure of the image of the section $\sigma$ is called the {\it Nash blow-up} of $X$,
and is denoted by $\widehat X$:
$$
\xymatrix@C=5pt{
& \PP_X(\wedge^n \Omega_X) \ar[d]^\pi
& \supseteq \; \widehat X := \overline{\sigma(X_{reg})} \quad\quad\quad  \\
X_{reg} \ar@{^{(}->}[r] \ar@/^15pt/[ur]^(.32)\sigma & X.
}
$$

\end{defn}

\begin{defn} Let $X$ be a variety.
We call a  morphism $\varphi: Y\to X$ a {\it partial resolution}, if $\varphi$ is proper birational
and $Y$ is normal. We also sometimes call $Y$ a partial resolution.
A prime divisor over $X$ is  a prime divisor that appears in a partial resolution of $X$.
A prime divisor $E$ over $X$ is called an {\it exceptional prime divisor} if a partial resolution on which $E$ appears is not isomorphic at the generic point of $E$.
\end{defn}

\begin{defn}
Let $E$ be a prime divisor over $X$, then the Mather discrepancy $\hke\in \ZZ_{\geq 0}$
and the Jacobian discrepancy $j_E$ at $E$
are defined as follows:

Let $\varphi: Y\to X$ be a partial resolution of $X$ such that $E$ appears on $Y$, 
 let $\widehat X\to X$ be the Nash  blow-up 
and let $Y\times_X\widehat{X}$ be the fiber product.
In the fiber product,  the union of the irreducible components each of which
dominates an irreducible component of $X$ is called the main part.

Replacing $Y$ by the normalization of the main part of the fiber product,
 we may assume that $\varphi:Y\to X$ factors through
the Nash blow-up $\widehat{X}\to X$.

By the universality of the Nash blow-up, the image $Im$ of the following homomorphism
is invertible:
$$\varphi^*(\wedge^d\Omega_X)\to \omega_Y,$$
where $d=\dim X$.
Restricting the above homomorphism to the smooth locus $Y_{sm}$, we can describe
$$Im|_{Y_{sm}}=\I \omega_{Y_{sm}}$$
with some invertible ideal sheaf $\I$,
as $\omega_{Y_{sm}}$ is invertible.
As $Y$ is normal, the generic point $\eta$ of $E$ is in ${Y_{sm}}$.
Let us define $$\hke= \val_E(\I)$$
and call it the Mather discrepancy of $X$ at $E$.
We note that if  $\I=\o_Y(-\widehat{K}_{Y/X})$ is expressed with a Cartier divisor $\widehat{K}_{Y/X}$ on $Y_{sm}$,
then $\hke$ is the coefficient of $\widehat{K}_{Y/X}$ at $E$.

For a coherent ideal sheaf $\a\subset \o_X$, let us define
$$\val_E(\a)=\min \{\val_E(f)\mid f\in \a\}.$$
Note that if $\varphi:Y\to X$ factors through the blow-up of $X$ by the ideal $\a$,
then $\val_E(\a)$ is the coefficient at $E$ of the divisor $Z$ defined by $\o_Y(-Z)=\a\o_Y$.
In particular, when $\a$ is the Jacobian ideal  $\jx$ of $X$,
we define
$$j_E=\val_E(\jx)$$
and call it the Jacobian discrepancy of $X$ at $E$.

Then, the difference $\hke-j_E$ is called the Mather--Jacobian (MJ) discrepancy of $X$ at $E$.
\end{defn}

\begin{defn}
Let $X$ be a variety over $k$, let $\a\subset \o_X$ be a non-zero coherent ideal sheaf, and let $n$ be a non-negative real number.
A real number
$$\amj(E; X, \a^n)=\hke-j_E-n\cdot\val_E(\a)+1$$
is called the MJ-log discrepancy of the pair $(X,\a^n) $ at $E$.

A pair $(X,\a^n) $ is called MJ-canonical (resp. MJ-log-canonical)
 at a (not necessarily closed) point $x\in X$,
if $$\amj(E; X, \a^n)\geq 1\ \ \ \  (\rm{resp. }\ \geq 0)$$
for every exceptional prime divisor $E$ over $X$ whose center contains $x$.

A pair $(X,\a^n) $ is called MJ-klt at a (not necessarily closed) point $x\in X$,
if $$\amj(E; X, \a^n)>0$$
for every  prime divisor $E$ over $X$ whose center on $X$ contains $x$.
Here, klt means ``Kawamata log-terminal" which appears in Minimal Model Program and
plays an important role (see, for example, \cite{koll}).
Our ``MJ-klt" is MJ-version of the usual klt.

We call a pair $(X,\a^n)$  MJ-canonical (resp. MJ-log-canonical, MJ-klt)
if it is MJ-canonical (resp. MJ-log-canonical, MJ-klt)  at every point of $X$.
\end{defn}
Here, MJ-klt is the MJ-version of the usual klt, i.e., Kawamata log terminal 
(see, for example, \cite{koll}).

\begin{rem} \label{comin}
\begin{enumerate}
\item
Note that the definitions of MJ-canonical and MJ-log-canonical require the condition only for
``exceptional" prime divisors, while the definition of MJ-klt requires the condition for all prime divisors over $X$.
However, regarding the MJ-log-canonical case,
if $(X, \a^n)$ is MJ-log-canonical at a point $x\in X$, then
$\amj(E; X, \a^n)\geq 0$  holds for every prime divisor over $X$ whose center contains
$x$.
\item If $X$ is normal and locally a complete intersection, the image of the canonical map $\wedge^d\Omega_X\to \omega_X$ is $\j_X\omega_X$
 (see, for example, \cite[Remark 9.6]{EM}).
 In this case $\hke-j_E=k_E $ for every prime divisor $E$ over $X$.
Therefore, the MJ-canonical and MJ-log-canonical cases are equivalent to the usual canonical and log-canonical cases, respectively.
\end{enumerate}
\end{rem}

\begin{defn}
\label{defmldmj}
 Let $X$ be a variety over $k$.
The Mather--Jacobian minimal log discrepancies (MJ-mld) of the pair
$(X, \a^n)$ at a proper closed subset $W\subset X$ and at a point $\eta\in X$
are defined as follows:

$${\mldmj}(W; X,\a^n)=\inf \{{\amj(E; X, \a^n)}\mid E \mbox{\ is\ a\ prime\ exceptional\ divisor}$$
$$\ \ \ \ \ \ \ \ \ \ \ \ \ \ \ \ \ \ \ \ \ \ \ \ \ \ \ \ \ \ \ \ \ \ \ \ \ \ \ \ \ \mbox{ \ over } X \mbox{with \ the\ center\ in\ }W\}$$

$${\mldmj}(\eta; X,\a^n)=\inf \{\amj(E; X, \a^n)\mid E \mbox{\ is\ a\ prime\ exceptional\ divisor}$$
$$\ \ \ \ \ \ \ \ \ \ \ \ \ \ \ \ \ \ \ \ \ \ \ \ \ \ \ \ \ \ \ \ \ \ \ \ \ \ \ \ \ \mbox{ \
over } X \mbox{\ with \ the\ center\  }\overline{\{\eta\}}\},$$
when $\dim X\geq2$.
When $\dim X=1$ and the right-hand side is $\geq 0$, then we define $\mldmj$ by the
right-hand side.
Otherwise, we define $\mldmj=-\infty$.
\end{defn}

\begin{rem}
\label{mld}

(i)  We strictly distinguish ``center in $Z$" from ``center $Z$."

(ii) For a point $x\in X$, the pair $(X,\a^n)$ is MJ-canonical (resp. MJ-log-canonical) if and only if
$\mldmj(\eta; X, \a^n)\geq 1$ (resp. $\mldmj(\eta; X, \a^n)\geq 0$) for every $\eta\in X$ such
that $x\in \overline{\{\eta\}}$.

(iii) When $\dim X\geq 2$, we can prove that  $\mldmj(W;X,\a^n)=-\infty$, \\
if
$\mldmj(W;X,\a^n)<0$.

(iv)
If $X$ is normal and of a complete intersection, then the MJ-log discrepancy coincides with the usual
log discrepancy.
Therefore, in this case, the MJ-canonical and MJ-log-canonical cases coincide with the usual canonical and log-canonical cases, respectively.
In particular, if $X$ is non-singular $\mldmj(W;X,\a^n)=\mld(W;X,\a^n)$,
where the right-hand side is the usual minimal log discrepancy.

\end{rem}

\vskip.5truecm
\section{The arc space and jet schemes of a variety}

\noindent
Throughout this section, $X$ is always a $d$-dimensional variety over an algebraically closed field $k$
of an arbitrary characteristic.
In this section, we prove the inversion of adjunction of minimal MJ-log discrepancies by
means of discussions of arc spaces and jet schemes.
This theorem has been independently proved in \cite{dd} and \cite{is}  based on the concept of \cite{EM}, when the base field is  of characteristic 0.
Here, we will present a characteristic-free proof for this theorem.

The basic tool is the arc space. 
So first we introduce arcs and jets of a variety.

\begin{defn}
 Let \( X \) be a scheme of finite type over \( k \) and a $K\supset k$ a field extension.
For  \( m\in \ZZ_{\geq 0} \) a \( k \)-morphism \( \spec K[t]/(t^{m+1})\to X \) is called an  {\it{\( m \)-jet}} of \( X \) and 
 \( k \)-morphism \( \spec K[[t]]\to X \) is called an {\it {arc}} of \( X \).
 The unique point of  \( \spec K[t]/(t^{m+1}) \) and   
 the closed point  of 
 \( \spec K[[t]] \) are both denoted by \( 0 \).
  and the image of $0$ by the jet or the arc is called the {\it center} of the jet or the arc.
\end{defn}
\begin{thm}
\label{existence} 
Let  \( X \) be a scheme of finite type over \( k \), \( {\mathcal S}ch/k \) the category of $k$-schemes and \( {\mathcal S}et \) the category of sets. A contravariant functor  \( F^X_{m}: {\mathcal S}ch/k \to {\mathcal S}et \) is defined as follows:
$$
 F^X_{m}(Z)=\Hom _{k}(Z\times_{\spec k}\spec k[t]/(t^{m+1}), X).
$$
  Then \( F^X_{m} \) is representable by a scheme  \( X_{m} \) of finite type over \( k \),
  therefore, there is a bijection as follows:
$$
 \Hom _{k}(Z, X_{m})\simeq\Hom _{k}(Z\times_{\spec k}
\spec k[t]/(t^{m+1}), X).
$$ 
   This scheme \( X_{m} \) is called the {\it space of \( m \)-jets} or  the {\it jet scheme}
   of  \( X \).

   There exists the projective limit $$X_\infty:=\lim_{\overleftarrow {m}} X_m$$
   and it is called the {\it space of arcs} or the {\it arc space} of $X$.
   There is also a bijection for a $k$-algebra A  as follows:
   $$
 \Hom _{k}(\spec A, X_{\infty})\simeq\Hom _{k}(\spec A[[t]], X).
$$

   \end{thm}

\begin{defn}
 For a variety $X$ over $k$, let $X_m$ $(m\in \NN)$ and $X_\infty$ be
 the $m$-jet scheme and the arc space of $X$.
 Denote the canonical truncation morphisms by
 $\psi_m: X_\infty\to X_m$ and $\pi_m: X_m\to X$.
 In particular we denote the extremal morphism  $\psi_0=\pi_\infty : X_\infty \to X$ by $\pi$.
 We also denote the canonical truncation morphism $X_{m'} \to X_m$ $(m'> m)$ by
 $\psi_{m', m}$.
 To specify the space $X$, we sometimes write $\psi^X_{m', m}$.

\end{defn}

\begin{defn} The pull-back \( \psi_{m}^{-1}(S)\subset X_\infty \) of a
  constructible set \(
  S\subset X_{m} \) $(m\in \ZZ_{\geq 0})$  is called a cylinder.

  A subset $C\subset X_\infty$ is called a thin set if there is a closed subset $Z\subset X$
  with dimension less than the dimension of $X$ such that $C\subset Z_\infty$.

  A subset $C\subset \psi_m^{-1}(S)$ is called an irreducible component of the
cylinder $\psi_m^{-1}(S)$ if it is a maximal irreducible closed subset of the cylinder.
Here, we note that a maximal irreducible closed subset of a cylinder exists 
because  the set consisting irreducible closed subsets in the cylinder is an inductive set with respect to
the inclusion order (the existence of a maximal irreducible component follows from  Zorn's lemma, which we always assume).
If $\cha\ k=0$
or $\dim X \leq 3$, then every cylinder in $X_\infty$ has only a finite number of irreducible components
(a resolution of singularities \cite{CP} is used to prove this result).
\end{defn}

\begin{defn}
For an arc $\gamma\in X_\infty$, 
the order of an ideal $\a\subset \o_X$ measured by $\gamma$
is defined as follows:
let $\gamma^*: \o_{X, \gamma(0)} \to k[[t]]$
be the corresponding
ring homomorphism of $\gamma$.
Then, we define
$$\ord_\gamma(\a)=\sup \{ r\in \ZZ_{\geq 0}\mid \gamma^*(\a)\subset (t^r)\},$$
We define the subsets ``contact loci" in the arc space as follows:
  $$\cont^m(\a)=\{\gamma \in X_\infty \mid \ord_\gamma(\a)=m\}$$
In a similar manner, we define
 $$\cont^{\geq m}(\a)=\{\gamma \in X_\infty \mid \ord_\gamma(\a)\geq m\}$$
By this definition, we can see that
$$\cont^{\geq m}(\a)=\psi_{m-1}^{-1}(Z(\a)_{m-1}),$$
Here, $Z(\a)$ is the closed subscheme defined by the ideal $\a$ in $X$.
Therefore, the contact loci are cylinders in $X_\infty$.

\end{defn}

\begin{defn}
 Let $E$ be a prime divisor over $X$ and
 $\varphi: Y\to X$  be a partial resolution of $X$ on which $E$ appears.
Let $p\in E$ be the generic point.
We define  $C_X(\val_E)=\overline{\varphi_\infty\left((\pi_\infty^Y)^{-1}(p)\right)}$,
also denoted by $N_E$ in the literature on the Nash problem.
Furthermore, for $q\in {\Bbb N}$, let $p_{q-1}\in (\pi_{q-1}^Y)^{-1}((E\cap Y_{reg})_{q-1})$ be the generic point.
Then, 
we define $$C_X(q\cdot\val_E)=\overline{\varphi_\infty(\psi_{q-1}^Y)^{-1}(p_{q-1})}$$
and call it the maximal divisorial set corresponding to the divisorial valuation
$q\cdot\val_E$. This definition is the same as that in \cite{DEI} and \cite{imax} in case $\cha\ k=0$.

An irreducible closed subset $C\subset X_\infty$ is called a divisorial set, if there exist $q\in \NN$
and a prime divisor $E$ over $X$ such that
the divisorial valuation $q\cdot\val_E$ is given by $\ord_\alpha$, where $\alpha$ is the generic point of $C$.

the generic point $\alpha \in C$ gives a divisorial valuation $q\cdot\val_E$ by $\ord_\alpha$.
A maximal divisorial set $C_X(q\cdot\val_E)$ is the maximal set among all divisorial sets corresponding to the valuation $q\cdot\val _E$.
\end{defn}

\begin{prop}[{\cite[Lemma 4.1]{DL}},
{\cite[before lemma 3.2]{Re1}} or {\cite[3.4]{Re2}},{\cite[Proposition 4.1]{EM}}]
\label{4.1}
For $e\in \ZZ_{\geq 0}$, the following subsets in $X_\infty$ and in $X_m$ are defined by
using the Jacobian ideal $\jx$:
 $$X_\infty^e:=\{\gamma\in X_\infty\mid ord_\gamma(\jx)= e\}\ \ \ \ \mbox{ and}
\ \ \ \
X_{m,\infty}^e:=\psi_m(X_\infty^e).$$
Then, there exists $c\in \NN\setminus \{0\}$ such that  for an integer $m\geq ce$ the canonical map $X_{m+1,\infty}^e\to X_{m,\infty}^e$ is a piecewise
trivial fibration with fibers isomorphic to $\AA^d$, where $d=\dim X$.
Therefore, we also know
that $X_{m+1,\infty}^{\leq e}\to X_{m,\infty}^{\leq e}$ is a piecewise
trivial fibration with fibers isomorphic to $\AA^d$, where we define
$$X_\infty^{\leq e}:=\{\gamma\in X_\infty\mid ord_\gamma(\jx)\leq e\}\ \ \ \ \mbox{ and}
\ \ \ \
X_{m,\infty}^{\leq e}:=\psi_m(X_\infty^{\leq e}).$$
\end{prop}

This proposition is stated in \cite{DL} under the condition that $\cha\ k=0$; however in
\cite{Re1} and \cite{EM}, 
its proof was also confirmed
to apply to perfect fields in the positive-characteristic case.\\

\begin{rem}
Using Proposition \ref{4.1}, {\it stable points} of the space of arcs were defined as follows  in \cite{Re2}: let $\gamma$ be a point of $X_\infty$ (i.e., $\gamma$ is a prime ideal of ${\cal O}_{X_\infty}$) and let $Z(\gamma)$be the set of zeros of $\gamma$ on $X_\infty$; then, $\gamma$ is a stable point if there exists $m_0 \in \NN$ and $G \in {\cal O}_{X_\infty}$, $G \not \in \gamma$, $G \in {\cal O}_{X_{m_0}}$, such that, for $m \geq m_0$, the map $\psi_{m+1,m}: X_{m+1} \rightarrow X_m$ induces a trivial fibration $\overline{\psi_{m+1}(Z(\gamma))} \cap (X_{m+1})_G \rightarrow \overline{\psi_{m}(Z(\gamma))} \cap (X_{m})_G$ with fiber $\AA^d$, where $(X_m)_G$ is the open subset $X_m \setminus Z(G)$ of $X_m$.\\
\end{rem}

The following is an easy consequence of the previous proposition when $char\ k=0$,
since in that case a cylinder has only finite number of irreducible components (\cite[Proposition 3.6]{DEI}), and
therefore, an irreducible component of a cylinder is regarded generically as a cylinder.
Herein, we give a proof that does not use the finiteness of the irreducible components of
a cylinder; thus, this proof also works  for a positive characteristic case.

\begin{lem} For an irreducible component $C$  of a cylinder in $X_\infty$ such that $C\not\subset \sing(X)_\infty $, its generic point $\gamma$ is a stable point of $X_\infty$. In particular,
there exists $e$ such that
$$
{C}_m^{\le e} := \psi_m(C) \cap X_{m,\infty}^{\le e}
$$
is a nonempty open subset of $\psi_m(C)$ and the codimension of ${C}_m^{\le e}$
inside $X_{m,\infty}^{\le e}$ stabilizes for $m\gg e$.
\end{lem}

\begin{proof} This lemma is proved in the same way as in the case of characteristic 0.
Let $C$ be an irreducible component of a cylinder $\Gamma$ in $X_\infty$ such that $C\not\subset \sing(X)_\infty $.
Let $\gamma\in C$ be the generic point. Then by $\gamma\not\in (\sing X)_\infty$,
 we have
$\ord_\gamma (\jx)=e > 0$. Then, ${C}_m^{\le e}$ is a nonempty open subset of $\psi_m(C)$. From this and Proposition 3.5, the result follows.
\end{proof}

Here, we show a result about a stable point, for the readers who are interested in this direction.

\begin{prop}
If $\gamma$ is a stable point of $X_\infty$, then there exists an irreducible cylinder $C \subset X_\infty$ whose generic point is $\gamma$. 
\end{prop}

\begin{proof}  We may assume that $X$ is affine. Let  $\gamma$ be a stable point of $X_\infty$. Then, by (\cite[Theorem 4.1]{Re1}) (see also the finiteness property of the stable points in \cite[3.10]{Re2}), there exist $G \in {\cal
O}_{X_\infty}$, $G \not \in \gamma$ and $L_1, \ldots, L_r \in {\cal
O}_{X_\infty}$ such that
\begin{equation}\label{eq1}
\gamma  \left({\cal O}_{(X_\infty)_\text{red}} \right)_G \ = \ (L_1, \ldots, L_r) \ \left({\cal O}_{(X_\infty)_\text{red}} \right)_G .
\end{equation}
Since ${\cal O}_{X_\infty} = \lim_{\rightarrow} {\cal O}_{X_n}$, there exists $n_0 \in \NN$ such that $G, L_1, \ldots, L_r \in {\cal O}_{X_{n_0}}$. Let us consider the constructible subset 
$$
S:= Z(L_1, \ldots , L_r ) \cap (G \neq 0) \ \subset \ {X_{n_0}}.
$$
Then $C = \psi_{n_0}^{-1} (S)$ is an irreducible cylinder whose generic point is $\gamma$ (see (\ref{eq1}) and  the definition of cylinder in \cite[beginning of section 5]{EM}).
\end{proof}

\begin{defn}  For a variety $X$, let $C$ be an irreducible component of a cylinder in $X_\infty$
such that $C\not\subset (\sing X)_\infty$,
then we define the codimension of $C$ in $X_\infty$ as follows:
$$\codim(C, X_\infty):=\codim(C_{m}^{\leq e}, X_{m,\infty}^{\le e})$$
for $m\gg e$, where $e=\ord_\gamma (\jx) $ for a generic point $\gamma\in C$.

Let $\Gamma\subset X_\infty$ be a cylinder not contained in $\sing(X)_\infty$,
then we define the codimension of the cylinder $\Gamma$ in $X_\infty$ as the minimal
value of the codimensions of the irreducible components not contained in $\sing(X)_\infty$.
\end{defn}

The following lemma was proved in \cite[Lemma 3.4]{DL} and  in \cite[Theorem 6.2,
Lemma 6.3]{EM}.
The former one was stated under the condition of characteristic 0, while the latter ones were stated under an arbitrary characteristic.
Note that the statement of Lemma 6.2 in \cite{EM} assumed the properness of the morphism $f$, but the properness was not actually used in the proof.

\begin{lem} [\cite{DL}, \cite{EM}]
\label{dl3.4}
Let $f:Y \to X$ be a birational morphism from a smooth variety $Y$ and
$f_m: Y_m\to X_m$ the induced morphism.
Let $\gamma\in Y_\infty$ be any arc such that $\tau:=\ord_\gamma(\widehat{K}_{Y/X})
< \infty$.
Then for 
$m\gg 0$, letting $\gamma_m=\psi_m^Y(\gamma)$, we have
$$f_m^{-1}(f_m(\gamma_m))\simeq \AA^\tau.$$
Moreover, for every $\gamma'_m\in f_{m}^{-1}(f_m(\gamma_m))$, we have $\psi^Y_{m,m-\tau}(\gamma_m)
=\psi^Y_{m,m-\tau}(\gamma'_m)$, where $\psi^Y_{m,m-\tau}:Y_m\to Y_{m-\tau}$ is
the canonical truncation morphism.
\end{lem}
\begin{thm}
\label{cod}
Let $E$ be a prime divisor over $X$ and $g:\tilde X\to X$ be a partial resolution of $X$
such that $E$ appears on $\tilde X$.
Then $C_X(q\cdot\val_E)$ is a cylinder of $X_\infty$ of codimension
$$\codim (C_X(q\cdot\val_E), X_\infty)=q\cdot (\val_E(\widehat K_{Y/X})+1).$$

\end{thm}
\begin{proof} Let $f:Y\to X$ be the restriction of $g$ on an open subset $Y\subset \tilde X$ such that $Y$ and $E\cap Y$ are smooth.
Then, by using Lemma \ref{dl3.4}, we can prove the required equality in the same way
as \cite[Theorem 3.9]{DEI}.
Note that  Theorem 3.9 in \cite{DEI} was stated under the condition that $f:Y\to X$ is
a resolution of the singularities of $X$, but 
the resolution is not actually needed for the proof
as the condition in Lemma \ref{dl3.4} does not require the properness of $f$.
\end{proof}

The following is a modified version  of Lemma \ref{dl3.4} for a blow-up $f$.

\begin{lem}
\label{blowup} Let $B\subset X$ be an irreducible reduced closed subset of
dimension $s$ with the defining ideal $I_B$ in $X$ and
let $\gamma\in X_\infty$ be an arc that is not contained in $\sing(X)_\infty\cup B_\infty$
and the center is a smooth closed point of $B$ with $e:=\ord_\gamma(I_B)>0$.
Let $f:Y\to X$ be the blow-up with the center $B$ and $\gamma'\in Y_\infty$ be the
lifting of $\gamma$.
Let $f_m: Y_{m,\infty}\to X_{m,\infty}$ be the morphism induced by $f$.
Then for  $m\gg 0$, letting $\gamma'_m=\psi_m^Y(\gamma')$, we have
$$\dim f_m^{-1}(f_m(\gamma'_m))\geq (d-s-1)e.$$
Moreover, for every $\gamma''_m\in f^{-1}_m(f_{m}(\gamma'_m))$ we have $\psi^Y_{m,m-e}(\gamma''_m)
=\psi^Y_{m,m-e}(\gamma'_m)$.

\end{lem}

\begin{proof} Since the problem is local, we may assume that $X$ is an affine scheme embedded in $A=\AA_k^N$,
$B$ is defined by the ideal $(x_1,\ldots, x_{N-s})$ in $A=\spec k[x_1,\ldots, x_N]$ and
the center of $\gamma$ is the origin $0=(0,\ldots, 0)\in B\subset X\subset A$.
Let $\gamma_m=\psi^X_m(\gamma)$, then by definition $f_m(\gamma'_m)=\gamma_m$.
Then, we can express the ring homomorphism $\gamma_m^*: k[x_1,\ldots, x_N]\to k[t]/(t^{m+1})$
corresponding to $\gamma_m$ as follows:

$$\gamma^*_m(x_i)=\sum_{j= e}^ma_i^{(j)}t^j \ \ \ (i=1,\ldots, N-s)\ \ \ \ \ \ \ \ $$
$$\gamma^*_m(x_i)=\sum_{j= 1}^ma_i^{(j)}t^j \ \ \ (i=N-s+1,\ldots, N),$$
where $a_i^{(j)}\in k$.
Here, by this assumption, we may assume that $a_1^{(e)}\neq 0$, without loss of generality.

Let $ f': A' \to A$ be the blow-up with the center $B$.
Then $A'$ is covered by $(N-s)$ affine spaces $$U(i)=\spec k[x_i, \frac{x_1}{x_i}, {\stackrel{i}{\stackrel{\vee}{\cdots}}}, \frac{x_{N-s}}{x_i}, x_{N-s+1}, \ldots, x_N]\ \ \  (i=1,\ldots, N-s).$$
By this assumption, we may assume that $\gamma'_m\in U(1)_m$.
We denote the restriction $f'_m|_{U(1)_m}: U(1)_m\to A_m$ of $ f'_m:A'_m\to A_m$
by the same symbol $f'_m$.
Let $y_i=\frac{x_i}{x_1}$ for $i=2,\ldots, N-s$, then $x_1,y_2,\ldots, y_{N-s}, x_{N-s+1},\ldots, x_N$ form a coordinate system of $U(1)$.
We can express $\gamma'_m\in U(1)_m$ as follows:
$${\gamma'_m}^*(x_1)=\sum_{j= e}^ma_1^{(j)}t^j\ \ \ \ \ \ \ \ \ \ \ \ \ \ \ \ \ \ \ \ \ \ \ \ \ \ \ \ \ \ \ \ \ $$
$${\gamma'_m}^*(y_i)=\sum_{j=0}^mb_i^{(j)}t^j\ \ \ \ (i=2, \ldots, N-s)\ \ \ \ \ \ \ $$
$${\gamma'_m}^*(x_i)=\sum_{j= 1}^ma_i^{(j)}t^j \ \ \ (i=N-s+1,\ldots, N), $$
where $b_i^{(j)}$'s satisfy ${\gamma'_m}^*(y_i){\gamma'_m}^*(x_1)={\gamma_m}^*(x_i)=
\sum_{j= e}^ma_i^{(j)}t^j $ for
$i=2,\ldots, N-s$.

Then, we can see that a jet $\alpha \in (f'_m)^{-1}(\gamma_m)$ is expressed by
$$\alpha^*(x_1)=\sum_{j= e}^ma_1^{(j)}t^j\ \ \ \ \ \ \ \ \ \ \ \ \ \ \ \ \ \ \ \ \ \ \ \ \ \ \ \ \ \ \ \ \ \ \ \ \ \ \ \ \ \ \ \ \ \ \ \ \ $$
$$\alpha^*(y_i)=\sum_{j=0}^{m-e}b_i^{(j)}t^j+\sum_{j=m-e+1}^m y_i^{(j)}t^j\ \ \ \ (i=2, \ldots, N-s)$$
$$\alpha^*(x_i)=\sum_{j= 1}^ma_i^{(j)}t^j \ \ \ (i=N-s+1,\ldots, N),\ \ \ \ \ \ \ \ \ \ \ \ \ \ $$
where $y_i^{(j)}$ $(i=2,\ldots, N-s, j=m-e+1,\ldots, m)$ can be an arbitrary element in $k$.
Therefore, we have that every $\alpha \in (f'_m)^{-1}(\gamma_m)$ is mapped to the same element
$\psi^{A'}_{m, m-e}(\gamma'_m)$ by the truncation morphism $\psi^{A'}_{m, m-e}$, in particular, which yields the second statement of the lemma.

Let $\gamma'_{m-e}=\psi^{A'}_{m,m-e}(\gamma'_m)$, then any element $\beta\in (\psi^{A'}_{m,m-e})^{-1}(\gamma'_{m-e})$ is expressed as

$$\beta^*(x_1)=\sum_{j= e}^{m-e}a_1^{(j)}t^j+\sum_{j=m-e+1}^mx_1^{(j)}t^j\ \ \ \  \ \ \ \ \ \ \ \ \ \ \ \ \ \ \ \ \ \ \ \ \ \ $$
$$\beta^*(y_i)=\sum_{j=0}^{m-e}b_i^{(j)}t^j+\sum_{j=m-e+1}^m y_i^{(j)}t^j\ \ \ \ (i=2, \ldots, N-s)$$
$$\beta^*(x_i)=\sum_{j= 1}^{m-e}a_i^{(j)}t^j+\sum_{j=m-e+1}^mx_i^{(j)}t^j \ \ \ (i=N-s+1,\ldots, N), $$
where $y_i^{(j)}$ and $x_i^{(j)}$  are arbitrary
elements of $k$.
Here, comparing the expressions of $\alpha^*$ and $\beta^*$ above, we obtain that
$$(f'_m)^{-1}(\gamma_m)=(f'_m)^{-1}(\gamma_m)\cap (\psi^{A'}_{m,m-e})^{-1}(\gamma'_{m-e})
\subset (\psi^{A'}_{m,m-e})^{-1}(\gamma'_{m-e})$$
is defined by $(s+1)e$ equations:
$$x_1^{(j)}=a_1^{(j)}, (j=m-e+1,\ldots, m)\ \ \ \ \ \ \ \ \ \ \ \ \ \ \ $$
$$x_i^{(j)}=a_i^{(j)}, (i=N-s+1,\ldots, N, \ \ j=m-e+1,\ldots, m)$$
in $ (\psi^{A'}_{m,m-e})^{-1}(\gamma'_{m-e})$.

Now, remember that  $Y$ is the strict transform of $X$ in $A'$, $\j_Y$ be the Jacobian ideal of $Y$
and let $r:=\ord_{\gamma'}(\j_Y)$.
Take $m$  such that it is sufficiently large with respect to $r$ and $e$.
Let $\rho_{m,m-1}:Y_{m, \infty}^{\leq r}\to Y_{m-1, \infty}^{\leq r}$ be the restriction of the
truncation morphism $\psi^Y_{m,m-1}: Y_m\to Y_{m-1}$.
Then, $\rho_{m, m-1}$ is a piecewise trivial morphism with the fiber $\AA^d$.
Therefore $\rho_{m,m-e}:Y_{m, \infty}^{\leq r}\to Y_{m-e, \infty}^{\leq r}$  has  exactly $de$-dimensional fibers.
Here we note that  $\rho_{m,m-e}^{-1}(\gamma'_{m-e})
\subset (\psi^{A'}_{m,m-e})^{-1}(\gamma'_{m-e})$.
By this and the above observation,
we obtain that
 $$(f'_m)^{-1}(\gamma_m)\cap \rho_{m,m-e}^{-1}(\gamma'_{m-e})\subset  \rho_{m,m-e}^{-1}(\gamma'_{m-e})$$
is defined by at most $(s+1)e$ equations in $de$-dimensional variety $\rho_{m,m-e}^{-1}(\gamma'_{m-e})$.
This implies that $f_m: Y_m\to X_m$ has a fiber $f_m^{-1}(\gamma_m)$ of dimension
$\geq (d-s-1)e$.

\end{proof}

\begin{rem} In the previous lemma,  the center of the lifting $\gamma'\in Y_\infty$ may be a singular point of $X$, while the center of $\gamma\in Y_\infty$ is non-singular in Lemma \ref{dl3.4}.
This, the situation of Lemma \ref{blowup} is different from that of Lemma \ref{dl3.4}.

As a corollary, we obtain the following.
But the result has been proved in \cite[Proposition 3.7 (vii)]{Re2} using the
concept of stable points in the space of arcs \cite[Definition 3.6]{Re2} (see Lemma 3.6).

  This corollary was also proved in \cite[Proposition 2.12]{DEI} for the characteristic 0 case by using resolutions of singularities.

\end{rem}

\begin{cor}
\label{maincor}
Every irreducible component $C$ of a cylinder such that $C\not\subset \sing(X)_\infty$
gives a divisorial valuation.
In particular, the closure of an irreducible component $C$ of finite intersection of contact loci $\cont^{m_i}(\a_i)$
$(i=1,\ldots, s)$ is a maximal divisorial set.
\end{cor}

\begin{proof}
  Let $\alpha$ be the generic point of $C$.
  First, we prove that there exists a prime divisor $E$ over $X$ such that the lifting $\tilde\alpha$
  of $\alpha$ onto $Y$ on which $E$ appears  has the center $E$.
  Let $B\subset X$ be the center of $\alpha$.
  If $\dim B=d-1$, then we have the required conclusion.
  Therefore, we assume that $\dim B=s<d-1$.
  By replacing $X$ by a small affine neighborhood, we may assume that $X$ is embedded in
  the affine space $A=\AA_k^N$.
  Let $f^{(1)}: A^{(1)}\to A$ be the blow-up with the center $B$.
  Then, $f^{(1)}$ is isomorphic outside $B$  and the image of $\alpha$ is not contained
  in $B$, therefore $\alpha$ is lifted onto the proper transform $Y^{(1)}$ of $X$ on $A^{(1)}$
  by the properness criteria applied to $f^{(1)}$.
  Then, the lifting $\alpha^{(1)}$ of $\alpha$ on $Y^{(1)}$ gives an irreducible component
  $C^{(1)}$ of a cylinder on $Y^{(1)}$.
  Indeed, if $C$ is an irreducible component of the cylinder $\Gamma=(\psi^X_m)^{-1}(S)$
  for a constructible subset $S\subset X_m$, then $C^{(1)}$ is an irreducible component of
  the cylinder $(\psi^{Y^{(1)}}_m)^{-1}\left((f^{(1)}_m)^{-1}(S)\right)$.

  Now, take $m$ to be sufficiently large and a general arc $\gamma\in C$ whose center 
  is a closed point in $B$. 
 Let  $e=\ord_\gamma I_B$.
 Then,  we obtain that the morphism $f^{(1)}: C^{(1)}_{m}\to C_m$ has relative dimension
  $\geq (d-s-1)e \geq 1$.
  Therefore, $\dim C^{(1)}_m> \dim C_m$, which yields $\codim (C^{(1)}, Y^{(1)}_\infty)<\codim (C, X_\infty)$.
  If the center of the generic point of $C^{(1)}$ is not of dimension $d-1$, we blow-up at the center and obtain the irreducible component whose codimension is less than the previous one.
  In this way, we obtain the successive blow-ups:
  $$A^{(n)}\to \cdots \to A^{(2)}\to A^{(1)}\to A$$
  with the sequence of irreducible components
  $$C^{(n)}, \cdots , C^{(2)},  C^{(1)}, C$$
  of  cylinders in $Y^{(n)}, \cdots , Y^{(2)},  Y^{(1)}, X$
  such that $$\codim C^{(n)}< \cdots <\codim C^{(2)}<\codim C^{(1)}<\codim C.$$
  Because the codimension is finite, this procedure should terminate, {\it i.e.,} there exists a number $n$ such that the dimension
  of the center of the generic point of $C^{(n)}$ has dimension $d-1$.
  This is the conclusion to our first claim.

  Then, we prove that the lifting $\alpha^{(n)}$ of $\alpha$ gives a divisorial valuation $q\cdot\val_E$.
  Replacing $Y^{(n)}$ by its normalization if necessary,
  we may assume that $Y^{(n)}$ is normal and the center of the lifting $\alpha^{(n)}$ of $\alpha$ on $Y^{(n)}$ is a prime divisor, say $E$.
  The arc $\alpha^{(n)}$ gives a ring homomorphism of local rings
  $\alpha^{(n)*}:\o_{Y^{(n)}, \alpha^{(n)}(0)}\to K[[t]]$, where $\alpha^{(n)}(0)$ is the center of $\alpha^{(n)}$ on $Y^{(n)}$ and $K$ is the residue field of $\alpha^{(n)}
  \in Y^{(n)}_\infty$.
  Here, we note that the center $\alpha^{(n)}(0)$  is the generic point of $E$.
  Then, we can see that $\alpha^{(n)*}$ factors through
  $$\beta:\widehat{\o}_{Y^{(n)}, \alpha^{(n)}(0)}=k(E)[[\tau]] \to K[[t]],$$
  where $k(E)$ is the rational function field of $E$ and $\tau$ is the generator of the maximal
  ideal of the local ring $\o_{Y^{(n)}, \alpha^{(n)}(0)}$.
  If we denote $\beta(\tau)=t^q$, then it implies that $\ord_{\alpha^{(n)}}=q\cdot\val_E$.

  Finally we assume that $C$ is an irreducible component of the intersection of $\cont ^{m_i}(\a_i)$'s.
  The generic point $\alpha$ of $C$ gives a divisorial valuation $q\cdot\val_E$ for a prime divisor over $X$
  by the above discussions.
  By definition of $\cont ^{m_i}(\a_i)$, it contains also the generic point of the maximal divisorial
  set $C(q\cdot\val_E)$ for every $i$.
  This complete the proof of the last assertion of the corollary.
\end{proof}

For the interpretation of MJ-minimal log discrepancies with center at a (not necessarily closed) point in terms of jet schemes,
we need the following definition:

\begin{defn}
\label{defeta}
Let $X$ be a variety and $\eta\in X$ a (not necessarily closed) point.
For a cylinder $S\subset X_\infty$ we define the codimension of $S\cap \pi_{}^{-1}(\eta)$ as follows:
$$\codim (S\cap \pi_{}^{-1}(\eta), X_\infty)
$$
$$
:=\inf\left\{\codim C\mid \begin{array}{l}C \mbox{\ is \ an\  irreducible \ component \ of\
}\ {S\cap\pi^{-1}(\overline{\{\eta\}})}\\
\mbox{dominating}\ \overline{\{\eta\}} \ \mbox{and\ not\ contained\ in}\
(\sing X)_\infty \ \end{array}\right\}.
$$
\end{defn}

By Corollary \ref{maincor}, we obtain the following interpretation of the MJ minimal
log discrepancy by the arc space.
This is proved \cite{is} for the characteristic zero case.

\begin{thm}
\label{formula}
Let  $(X, \a)$ be a pair consisting of an arbitrary variety $X$  and a non-zero coherent ideal sheaf
 $\a\subset \o_X$.
Let $n\geq 0$ be a real number, let $W$ be a proper closed subset of $X$ and let $I_W$ be the (reduced) ideal of $W$.
Then,
\vskip.5truecm
$\mldmj(W;X,\a^n)$
\begin{equation}
\label{hmld=}
=\inf_{m_i\in\NN}\{\codim(\cont^{m_1}(\a)\cap\cont^{m_2}(\j_X)\cap\cont^{\geq 1}(I_W))-m_1n-m_2\}
\end{equation}

$\mldmj(W;X,\a^n)$
\begin{equation}
\label{hmldgeq}
=\inf_{m_i\in\NN}\{\codim(\cont^{\geq m_1}(\a)\cap\cont^{\geq m_2}(\j_X)\cap\cont^{\geq 1}(I_W))-m_1n-m_2\}.
\end{equation}

For a (not necessarily closed) point $\eta\in X$, we also have

$\mldmj(\eta;X,\a^n)$
\begin{equation}
\label{hmldpt=}
=\inf_{m_i\in\NN}\{\codim(\cont^{m_1}(\a)\cap\cont^{m_2}(\j_X)\cap\pi^{-1}(\eta))-m_1n-m_2\}.
\end{equation}

$\mldmj(\eta;X,\a^n)$
\begin{equation}
\label{hmldptgeq}
=\inf_{m_i\in\NN}\{\codim(\cont^{\geq m_1}(\a)\cap\cont^{\geq m_2}(\j_X)\cap\pi^{-1}(\eta))-m_1n-m_2\}.
\end{equation}

\end{thm}

\begin{proof}
For $\cha\ k=0$, we proved these in
\cite[Proposition 3.8 and Remark 3.8]{is},  using \cite[Theorem 3.9]{DEI}.
The proof for positive characteristic case follows just in the same way  by using Theorem \ref{cod}  and Corollary \ref{maincor}.
However, for the reader's convenience, we  write down the proof here.
 For the proof of the equality (\ref{hmld=}), we will show first the following inequality:
 
 $\mldmj(\eta;X,\a^n)$
 
$\geq \inf_{m_i\in\NN}\{\codim(\cont^{m_1}(\a)\cap\cont^{m_2}(\j_X)\cap\pi^{-1}(\eta))-m_1n-m_2\}.$
 
Let $E$ be any prime divisor over
  $X$ with the center in $W$.
  Let $m_1=\val_E(\a)$, $m_2=\val_E(\j_X)$ and $v=\val_E$.
  Then, there is a non-empty open subset $C$ of the maximal divisorial set  $C_X(v)$ such that $C\subset \cont^{m_1}(\a)\cap\cont^{m_2}(\j_X)\cap\cont^{\geq 1}(I_W)$.
  Hence,
    $$\widehat k_E-\val_E(\j_X)-n\cdot\val_E(\a)+1=\codim(C_X(v))-m_1n-m_2$$
   $$\geq
    \codim(\cont^{m_1}(\a)\cap\cont^{m_2}(\j_X)\cap\cont^{\geq 1}(I_W))-m_1n-m_2,$$
    which yields the required inequality unless $\dim X=1$ and $\mldmj(X,\a^n)=-\infty$.

    When $\dim X=1$ and $\mldmj(X,\a^n)=-\infty$, there is a prime divisor $E$ over $X$ with the center in $W$
    such that $\widehat k_E-\val_E(\j_X)-n\cdot\val_E(\a)+1<0$.
    Let $m_1=\val_E(\a)$ and $m_2=\val_E(\j_X)$, then
    $\codim C_X(\val_E)-m_1n-m_2<0.$
    Here, for every $q\in \NN$, by Proposition \ref{cod},
    $$\codim C_X(q\cdot\val_E)-qm_1n-qm_2=q(\codim C_X(\val_E)-m_1n-m_2)<0.$$
    As a non-empty open subset of $C_X(q\cdot\val_E)$ is contained in \\
$\cont^{ qm_1}(\a)\cap\cont^{qm_2}(\j_X)\cap\cont^{\geq 1}(I_W)$,
    we have
    $$\codim(\cont^{ qm_1}(\a)\cap\cont^{qm_2}(\j_X)\cap\cont^{\geq 1}(I_W))-qm_1n-qm_2$$
    $$\leq \codim (C_X(q\cdot\val_E))-qm_1n-qm_2$$
    $$=q(\codim (C_X(\val_E))-m_1n-m_2)<0.$$
    Here, if $q\to \infty$, then, we have $$\codim(\cont^{ qm_1}(\a)\cap\cont^{qm_2}(\j_X)\cap\cont^{\geq 1}(I_W))-qm_1n-qm_2
    \to -\infty,$$ which implies that 
   the right-hand side of (\ref{hmld=}) in the theorem is $-\infty$.

    For the proof of the converse inequality,
    
     $\mldmj(\eta;X,\a^n)$
 
$\leq \inf_{m_i\in\NN}\{\codim(\cont^{m_1}(\a)\cap\cont^{m_2}(\j_X)\cap\pi^{-1}(\eta))-m_1n-m_2\},$
    we may assume that $\widehat k_E-\val_E(\j_X)-n\cdot\val_E(\a)+1\geq 0$
    for every prime divisor $E$ over $X$ with the center in $W$.
    Indeed if there is a prime divisor $E$ with the  center in $W$ and $\widehat k_E-\val_E(\j_X)-n\cdot\val_E(\a)+1<0$,
    then $\mldmj(W;X,\a^n)=-\infty$ by Remark \ref{mld}, (iii), and therefore, the required
    inequality is trivial.

    For $m_i\in \NN$, let $C\subset
    \cont^{m_1}(\a)\cap\cont^{m_2}(\j_X)\cap\cont^{\geq 1}(I_W)$ be an irreducible component
    that gives the codimension of the cylinder.
    Then, the closure $\overline C$ is $C_X(v)$ for some  divisorial valuation $v$
   by Corollary \ref{maincor}.
    Let $v=q\cdot\val_E$. 
    Then $m_1=v(\a)$, $m_2=v(\j_X)$ and $E$ is a prime divisor over $X$ with the center in $W$ and
    $$\codim(\cont^{m_1}(\a)\cap\cont^{m_2}(\j_X)\cap\cont^{\geq 1}(I_W))-m_1n-m_2$$
    $$=\codim(C_X(v))-m_1n-m_2$$
    $$=	q(\widehat k_E+1)-qn\cdot\val_E(\a)-q\cdot\val_E(\j_X)\geq \widehat k_E+1-n\cdot\val_E(\a),$$
    which yields the required inequality.

    For the proof of (\ref{hmldgeq}) of the theorem, let
    $$a_m=\codim(\cont^{m_1}(\a)\cap\cont^{m_2}(\j_X)\cap\cont^{\geq 1}(I_W))-m_1n-m_2,$$
$$b_m=\codim(\cont^{\geq m_1}(\a)\cap\cont^{\geq m_2}(\j_X)\cap\cont^{\geq 1}(I_W))-m_1n-m_2.$$
  As $\cont^m(\a)\subset \cont^{\geq m}(\a)$, we have $a_m\geq b_m$.
  Therefore, it follows\\ $\inf_m \{a_m\}\geq \inf _m\{b_m\}$.

  Next, we prove the converse inequality.
  For every $m\in \NN$, let $C_X(v)$ be the irreducible component of
  $\codim(\cont^{\geq m_1}(\a)\cap\cont^{\geq m_2}(\j_X)\cap\cont^{\geq 1}(I_W))$ that gives the codimension.
  Then, for $m'_1:=v(\a)\geq m_1$ and $m'_2:=v(\j_X)\geq m_2$ we have
  $$\codim(\cont^{\geq m_1}(\a)\cap\cont^{\geq m_2}(\j_X)\cap\cont^{\geq 1}(I_W))$$
  $$=
  \codim(\cont^{ m'_1}(\a)\cap\cont^{ m'_2}(\j_X)\cap\cont^{\geq 1}(I_W)).$$
  Hence, $b_m\geq a_{m'}$, which yields $\inf_m \{b_m\}\geq \inf _m\{a_m\}$.

The equalities  (\ref{hmldpt=}) and ({hmldptgeq}) follow in the same way, indeed one has only to
be careful to replace ``center in $W$" by ``center $\overline{\{\eta\}}$."

\end{proof}

\begin{rem}\label{extended}
 Our formula can be easily extended for the combination of ideals
 $\a_1,\a_2,\cdots,\a_r$ instead of one ideal $\a$. I.e., we have
 $$\mldmj(W;X,\a_1^{e_1}\a_2^{e_2}\cdots\a_r^{e_r})=$$
 $$\inf_{m_i\in\NN}\left\{\codim\left(\left(\bigcap_{i=1}^r\cont^{ m_i}(\a_i)\right)\cap\cont^{ m_{r+1}}(\j_X)\cap\cont^{\geq 1}(I_W)\right)
  -\sum_{i=1}^rm_ie_i-m_{r+1}\right\},$$
 where $e_i$'s are positive real numbers.
 Here, any of $\cont^{m_i}(\a_i)$'s can be replaced by $\cont^{\geq m_i}(\a_i)$.
 For  simplicity of the notation and the proofs, we keep formulating the forthcoming formulas
 for one ideal only.
 But note that the formulas in this section are also valid under this combination form.

 The description of log-canonical threshold  on a non-singular variety by jet schemes in positive characteristic case is obtained by Zhu \cite{zhu}.
\end{rem}

Here, we will show the inversion of adjunction Formula for MJ-minimal log discrepancies for
the base field of arbitrary characteristic.
The proof basically follows  the idea in the case of characteristic 0.
For the proof, we prepare some lemmas.
The first one was proved in \cite{EM} for an arbitrary characteristic.

\begin{lem}[{\cite[Lemma 8.4]{EM}}]
\label{ci}
 Let $A$ be a non-singular variety and $M=H_1\cap\cdots
\cap H_c$ a codimension $c$ complete intersection in $A$.
If $C$ is an irreducible locally closed cylinder in $A_\infty$ such that
$$C\subset \bigcap_{i=1}^c\cont ^{\geq d_i}(H_i),$$
and if there is an arc $\gamma\in C\cap M_\infty$ with $\ord_\gamma(\j_M)=e$,
then
$$\codim(C\cap M_\infty, M_\infty)\leq \codim(C, A_\infty)+e-\sum_{i=1}^c d_i.$$
\end{lem}

\begin{say}
\label{action}
 Consider \( {\Bbb{G}}_m=\AA^1\setminus \{0\}=\spec k[s,s^{-1}] \) as a multiplicative 
  group scheme. 
  For \( m\in \NN\cup\{\infty\} \), the morphism \( k[[t]] \to 
  k[s,s^{-1}][[ t]] \) defined by \( t\mapsto s\cdot t \) 
  gives an action  \[ \phi: {\Bbb{G}}_m \times_{\spec k} \spec k[[t]]\to \spec k[[t]] \] of 
  \( {\Bbb{G}}_m  \) on \( \spec k[[t]] \).
  Therefore, it gives an action 
  \[ \Phi: {\Bbb{G}}_m\times_{\spec k}X_\infty\to X_\infty \]
  of \( {\Bbb{G}}_m \) on \( X_\infty \).
  As \( \phi \) is extended to a morphism:
  \(\overline {\phi}:\AA^1\times _{\spec k} k[[t]]\to k[[t]]  \), we 
  obtain the extension 
   \[\overline {\Phi}: \AA^1\times_{\spec k}X_\infty\to X_\infty \]
   of \( \Phi \).
   
   In a similar way, we have an action of ${\Bbb{G}}_m$ on $X_r$ for $r\in \NN$ and
   note that the truncation morphisms are ${\Bbb{G}}_m$-equivariant.

The following lemma is proved in Lemma 8.3 in \cite{EM} for characteristic 0 case by using a resolution
of singularities.
To give a characteristic free proof,
we have only to check that the proof in \cite{EM} works under the normalized blow-up
instead of a resolution.
Let $\overline\Phi: \AA^1\times X_\infty \to X_\infty$
be the morphism given as above.
Note that if an arc $\gamma$ has the center $x\in X$, then $\Phi(0,\gamma)$ is the
constant arc over $x$.
\end{say}

\begin{lem}[\cite{EM}]
\label{lemma8.3}
Let $C$ be a non-empty cylinder on $X_\infty$ for a variety $X$.
If $\overline\Phi(\AA^1\times C)\subset C$, then $C\not\subset (\sing X)_\infty$.
\end{lem}

\begin{proof} Take an arc $\gamma\in C$ and denote $C=(\psi^X_m)^{-1}(S)$ for some
constructible set
$S\subset X_m$.
Let $x\in X$ be the center of $\gamma$, then by $\AA^1$-invariance of $C$,
the trivial $m$-jet ${\bold 0}_m^x$ at $x$ belongs to $S$.
Let $f:X'\to X$ be the normalized blow-up at the point $x$ and take a non-singular point $x'$ of $X'$ inside $ f^{-1}(x)$.
Let $C'$ be the cylinder $(\psi^{X'}_m)^{-1}({\bold 0}_m^{x'})$,
then $f_\infty(C')\subset C$.
As $C'$ is not a thin set in $X'_\infty$, $C$ is not thin in $X_\infty$ either, which gives that $C\not\subset (\sing X)_\infty$.

\end{proof}

The following is the inversion of adjunction formula for MJ-minimal log discrepancies which was proved for a characteristic 0 base field  in \cite{dd} and
\cite{is} independently.
We give here a proof for the arbitrary characteristic case.
We prepared necessary lemmas for the proof in arbitrary characteristic case 
that were used in the proof in characteristic 0 case.
So the outline of the proof can be the same as that of the characteristic 0 case (presented in  \cite{dd} and \cite{is} which are based on \cite{EM}).
Here, we remind the reader that if $\dim X=1$ and there is a  prime divisor $E$  over
$X$ 
such that $a_{MJ}(E; X, \a^n)<0$, then we define $\mldmj(x; X, \a^n)= -\infty$ 
(see, Definition \ref{defmldmj}).

\begin{thm}[inversion of adjunction]
\label{ia}
Let $X$ be a variety over an algebraically
closed field $k$ of an arbitrary characteristic and $A$ a smooth variety containing $X$ as a closed
subscheme of codimension $c$.
 Let $\widetilde\a\subset \o_A$ be a coherent ideal sheaf such that its image ${\a}:=\widetilde\a\o_X\subset \o_X$ is non-zero.
 Denote the ideal of $X$ in $A$ by $I_X$.
Then, for a proper closed subset  $W$ of  $X$, we have 
\begin{equation}
\label{iaclosed}
\mldmj(W; X,{\a}^n)=\mld(W;A,\widetilde\a^n I_X^c).
\end{equation}
For a point $\eta\in X$, we have 
\begin{equation}
\label{iapoint}
\mldmj(\eta; X,{\a}^n)=\mld(\eta;A,\widetilde\a^n I_X^c).
\end{equation}
\end{thm}

\begin{proof}
  First, we prove the inequality $\geq$ in (\ref{iaclosed}).
  We assume the contrary and  will induce a
  contradiction.
  By the assumption, there exist $e, m\in \NN$
  and an irreducible component $C\subset \cont^{\geq m}(\a)\cap\cont^e(\j_X)\cap\cont^{\geq 1}(I_W)$ such that $\codim C-mn<\mld(W;A,\widetilde\a^n I_X^c)$.
 Then,  for a sufficiently large $s\in \NN$, we obtain
\begin{equation}
\label{small}
(s+1)d-\dim \psi_m(C)-mn=\codim C-mn<\mld(W;A,\widetilde\a^n I_X^c),
\end{equation}

where $I_W$ is the defining ideal of $W$ in $X$.
As $\psi_s(C)\subset \cont^{\geq m}(\a)_s\cap\cont^{\geq 1}(I_W)_s=
\cont^{\geq m}(\widetilde\a)_s\cap\cont^{\geq 1}(\widetilde I_W)_s \cap X_s,$
where $\widetilde I_W$ is the defining ideal of $W$ in $\o_A$,
we have
$$C\subset (\psi_s^A)^{-1}(\psi_s(C))\subset
\cont^{\geq m}(\widetilde\a)\cap\cont^{\geq 1}(\widetilde I_W)\cap\cont^{\geq s+1}(I_X)=:S,$$
where $\psi^A_s: A_\infty \to A_s$ is the truncation morphism.

Now we obtain
$$(d+c)(s+1)-\dim \psi_s(C)=\codim (\psi_s(C), A_s)=\codim ((\psi_s^A)^{-1}(\psi_s(C), A_\infty)$$
$$\geq \codim (S, A_\infty) \geq \mld(W;A,\widetilde\a^n I_X^c) +mn+c(s+1),$$
which is a contradiction to (\ref{small}).

For the converse inequality in (\ref{iaclosed}), we have only to show the following claim

\begin{claim}
\label{E&F}
  For a prime divisor $F$ over $A$ with center in $W$, there is a prime divisor $E$ over $X$ with
  center in  $W$ and an integer $q\geq 1$ such that
  \begin{equation}
  \label{claim}
  q\cdot\amj(E; X, \a^n)\leq \amj(F, A, \tilde\a^nI_X^c)
\end{equation}
\end{claim}

We can prove that this induces the inequality $$\mldmj(W; X,{\a}^n)\leq\mld(W;A,\widetilde\a^n I_X^c)$$ as
follows:
If $\mldmj(W; X,{\a}^n)=-\infty$, then the required inequality is trivial.
If $\mldmj(W; X,{\a}^n)\geq 0$ then $\amj(E; X, \a^n)\geq 0$ in (\ref{claim}),
which yields $\amj(F, A, \tilde\a^nI_X^c)\geq \frac{1}{q}\amj(F, A, \tilde\a^nI_X^c)\geq
\amj(E; X, \a^n)\geq \mldmj(W; X,{\a}^n)$ for every prime divisor $F$ over $A$ with center in
$W$.

In the following, we prove Claim \ref{E&F}.
Consider the maximal divisorial set
$$V=C_A(\val_F)\subset A_\infty;$$
then we have
$$\codim (V, A_\infty)=k_F +1.$$
The intersection $V\cap X_\infty\subset X_\infty$ is a non-empty cylinder in $X_\infty$ and is not
contained in $(\sing X)_\infty$ by Lemma \ref{lemma8.3}.
Let $C$ be an irreducible component of $V\cap X_\infty$ that is not contained in the arc space of the singular locus of $X$, and the generic point $\gamma\in C$ gives the minimal value
$\ord_\gamma(\j_X)=e$ among the points in $V\cap X_\infty$.
We can also assume that $C$ has the minimal codimension among the components with
$\ord_\gamma(\j_X)=e$.
Then, by Corollary \ref{maincor}, there exists a prime divisor $E$ over $X$ with the center in $W$, 
such that the generic point $\gamma$ of $C$ gives the divisorial valuation $q\cdot\val_E$ for some $q\in \NN$.
Therefore, $C\subset C_X(q\cdot\val_E)$ and  we have the following inequalities:
\begin{equation}
\label{ia2}
q\cdot(\hk_E+1)=\codim (C_X(q\cdot\val_E), X_\infty)\leq \codim (C, X_\infty)
\end{equation}

Now, by an appropriate choice of  $c$ generators of  $I_X$, we can take a complete intersection
scheme $M$ of dimension $d$ containing $X$ such that
$\ord_\gamma(\j_M)=e$.
Then, by \cite[Corollary 9.2]{EM}, we have $$\j_M\cdot\o_X\subset ((I_M:I_X)+I_X)/I_X.$$
It follows that $\gamma$ belongs to the open cylinder
$$V_0:=V\cap \cont^{\leq e}(\j_M)\cap\cont^{\leq e}(I_M:I_X).$$
Here, we note that
$$ V_0\cap X_\infty=V_0\cap M_\infty.$$
Indeed, by the definition of $V_0$ an arc $\alpha\in V_0\cap M_\infty$ has finite order along $(I_M:I_X)$ which is a
defining ideal of the union $X'$ of the components of $M$ other than $X$.
This  means
$\alpha\not\in X'_\infty$.

Here, for every $\beta\in M_\infty$, we note that $\ord_\beta(\j_X)\leq \ord_\beta(\j_M)$; 
therefore, by the definition of $C$, we obtain
\begin{equation}
\label{ia3}
\codim (C, X_\infty)=\codim (V_0\cap X_\infty, X_\infty)=\codim (V_0\cap M_\infty, M_\infty).
\end{equation}

Apply Lemma \ref{ci} to $V_0$ and we obtain
$$\codim (V_0\cap M_\infty, M_\infty)\leq \codim (V_0, A_\infty)+e -\sum_{i=1}^c d_i,$$
where $d_i$'s satisfy  $V_0\subset \bigcap_{i=1}^c\cont^{\geq d_i}(H_i)$ for $M=H_1\cap\cdots\cap H_c$.
Then, as $V_0$ is an open subset of $V=C_A(\val_F)$, the term $\sum d_i $ in the above inequality satisfies $\sum d_i \geq c\cdot \val_F(I_X)$ and the equality
$\codim (V_0, A_\infty)=\codim (V, A_\infty)$ holds.
On the other hand, $e=\ord_\gamma(\j_X)=q\cdot \val_E(\j_X)=q\cdot j_E$.
Therefore, we have the following inequality
\begin{equation}
\label{ia4}
\codim (V_0\cap M_\infty, M_\infty)\leq \codim (V_0, A_\infty)+q\cdot\val_E(\j_X) -c\cdot \val_F(I_X).
\end{equation}

Now combining (\ref{ia2}), (\ref{ia3}), and (\ref{ia4}), we obtain
$$q\cdot(\hk_E-j_E+1)\leq \codim (V, A_\infty)-c\cdot\val_F(I_X). $$

Note that for any proper coherent ideal sheaf $\b\subset \o_M$ not vanishing on any component of $X$, we have $\val_C(\b|_X)\geq \val_V(\b)$ by the inclusion $C\subset V$.
In particular, this implies that
$$q\cdot\val_E(\b|_X)\geq \val_F(\b),$$
which yields the inequality in Claim \ref{E&F}.

The proof of the equality (\ref{iapoint}) follows in the same way.
Indeed, we have only to be careful to replace
``with the center in $W$" by ``with the center  $\overline{\{\eta\}}$" in the proof above.

\end{proof}

\begin{rem} The formula in the theorem  can be easily extended for the combination of ideals
 $\a_1,\a_2,\cdots,\a_r$ instead of one ideal $\a$. I.e., we have
 $$\mldmj(W;X,\a_1^{e_1}\a_2^{e_2}\cdots\a_r^{e_r})
={\mld}(W;A,\tilde\a_1^{e_1}\tilde \a_2^{e_2}\cdots\tilde\a_r^{e_r}I_X^c),$$
where $e_i$'s are non-negative real numbers and $\tilde\a_1,\tilde\a_2,\cdots,\tilde\a_r$ are coherent ideal sheaves of $\o_A$ such that
the extensions $\a_i=\tilde\a_i\o_X$'s are not zero.

\end{rem}

\begin{cor}
\label{useful}
Let $\eta\in X$ be a point of a variety $X$ of dimension $d$.
For each $m\in \NN$ we denote $r_m=\dim \pi_m^{-1}(x)$ for a general closed point
$x\in \overline{\{\eta\}}$.
Then, we have the equality:
\begin{equation}
\label{basicequation}
\mldmj(\eta; X):=\mldmj(x; X, \o_X)=\inf_m\left\{(m+1)d-\dim \overline{\pi_m^{-1}(\eta)}\right\}
\end{equation}
$$
= \inf_m\left\{(m+1)d-(\dim  \overline{\{\eta\}}+r_m)\right\}.$$
\end{cor}

\begin{proof}
Since the statement is local, we can assume that $X$ is a closed subvariety of codimension
$c$ of a non-singular
variety $A$.
By Theorem \ref{ia}, we have
$$\mldmj(\eta; X)=\mld(\eta;A, I_X^c)$$
and by Theorem \ref{formula},
we have
$$\mld(\eta;A, I_X^c)=\inf_{m\geq 0} \left\{\codim (\cont^{m+1}(I_X)\cap (\pi^A)^{-1}(\eta), A_\infty)- (m+1)c
\right\},$$
where, for our convenience  later, we shift $m$ to $m+1$ on the right-hand side.
By Definition \ref{defeta}, $\codim (\cont^{m+1}(I_X)\cap(\pi^A)^{-1}(\eta), A_\infty)$ is the minimal
codimension of irreducible components $C\subset (\psi_m^A)^{-1}(X_m)\cap \pi^{-1}(\overline{\{
\eta\}})$ in $A_\infty$ that dominate $\overline{\{
\eta\}}$.
Therefore,
$$\codim (\cont^{m+1}(I_X)\cap (\pi^A)^{-1}(\eta), A_\infty)=\codim \left( \overline{(\pi_m^X)^{-1}(\eta)}, A_m \right).$$
As $\codim \left( \overline{(\pi_m^X)^{-1}(\eta)}, A_m \right)= (m+1)(d+c)-\dim \overline{(\pi_m^X)^{-1}(\eta)}$, we have the first equality in the corollary.
For the second equality, we need to know that
$$\dim \overline{\pi_m^{-1}(\eta)}=\dim \overline{\{\eta\}}+r_m.$$

\end{proof}

\begin{cor}
\label{very}
Let $X$ be a variety of dimension $d$ and $\a$ a coherent ideal sheaf of $\o_X$.
Let $V\subset W$ be two irreducible proper closed subsets of $X$ and
$\eta_V$ and $\eta_W$ are the generic points of $V$ and $W$, respectively.
Then, 
\begin{enumerate}
\item[(i)] We have the following inequality:
$$\mldmj(\eta_V; X, \a)\leq \codim (V, X),$$
where the equality holds if and only if $\eta_V$ is a regular point and $\a=\o_X$ around $\eta_V$.

In particular, if $x\in X$ is a closed point, then
$$\mldmj(x;X,\a)\leq d$$
and the equality holds if and only if $X$ is smooth at $x$ and $\a=\o_X$ around $x$.
(This statement for the usual $\mld$ is Shokurov's conjecture and is not yet proved.)
\item[(ii)] $\mldmj(\eta_V; X, \a)\leq \mldmj(\eta_W; X, \a)+\codim(V, W),$\\
where the equality holds for very general $V$ in $W$;
i.e.,  $\eta_V$ is in the complement of a countable number of closed  subsets in $W$.
(If $k$ is an uncountable field, then this subset is non-empty.)
\item[(iii)] If $\cha\ k=0$, then the equality in (ii) holds for general $V$ in $W$.
\end{enumerate}
\end{cor}

\begin{proof} First note that if $\a\neq \o_X$ around $\eta_V$, then
$$\mldmj(\eta_V; X, \a)< \mldmj(\eta_V; X, \o_X)$$
by the definition of MJ-log discrepancy.
Here, by Corollary \ref{useful}, we have
\begin{equation}
\label{use2}
\mldmj(\eta_V;X, \o_X)\leq (m+1)d -r_m-\dim V
\end{equation}

for every $m\geq 0$.
Therefore, in particular for $m=0$, we obtain
$$\mldmj(\eta_V;X, \o_X)\leq \codim (V, X).$$

Now assume the equality in (i): $\mldmj(\eta_V; X,\o_X)=d-\dim V$.
then, by the first comment in this proof, we have $\a=\o_X$.
Consider the inequality (\ref{use2}) for $m=1$,
we obtain
\begin{equation}
\label{use3}
d-\dim V\leq (d-\dim V) +d-r_1.
\end{equation}
Here, we note that $r_1$ is the dimension of the tangent space of $X$ at a
general closed point $x\in V$.
Then, the  inequality (\ref{use3}) gives $r_1=d$, which means that  general points of V are non-singular in $X$.

For the proof of (ii), let $$s_{V}(m,n):= \codim\left(\cont^{\geq m}(I_X)\cap\cont^{\geq n}(\a)\cap(\psi^A)^{-1}(\eta_V), A_\infty\right)-mc-n.$$

$$s_{W}(m,n):= \codim\left(\cont^{\geq m}(I_X)\cap\cont^{\geq n}(\a)\cap(\psi^A)^{-1}(\eta_W), A_\infty\right)-mc-n.$$
Then, by Theorem \ref{formula} and Theorem \ref{ia}, we have
$$\mldmj(\eta_V;X, \a)=\inf_{m,n}s_{V}(m,n) \ \mbox{and}\ \ \mldmj(\eta_W;X, \a)=\inf_{m,n}s_{W}(m,n).$$

Remember the action of ${\Bbb G}_m$ on $A_r$ (\ref{action}).
Then, for each $m, n$, by an appropriate $r=r(m,n)\in \NN$ and a ${\Bbb G}_m$-invariant closed subset 
$$S_{m,n}=\psi_r(\left(\cont^{\geq m}(I_X)\cap\cont^{\geq n}(\a)\right)\subset A_r,$$
 we can express
$$s_V(m,n)-s_W(m,n)=\dim S_{m,n}\cap \pi_r^{-1}(\eta_W)-\dim S_{m, n}\cap \pi_r^{-1}(\eta_V).$$
$$=(\dim W+\delta_W) -(\dim V+\delta_V),$$
where $\delta_V$ and $\delta_W$ are the dimensions of general fibers of $\pi_r|_{S_{m,n}\cap\pi_r^{-1}(V)}$ and  $\pi_r|_{S_{m,n}\cap\pi_r^{-1}(W)}$, respectively.
As $S_{m,n}$ is  ${\Bbb G}_m$-invariant,  the restricted morphism $\pi_r|_{S_{m,n}\setminus \sigma(A)}:
{S_{m,n}\setminus \sigma(A)}
\to A $
factors through the projective morphism $\pi'_r: ({S_{m,n}\setminus \sigma(A)})/ {\Bbb G}_m\to A$.
Here, $\sigma(A)\subset A_r$ is the subset consisting of the trivial $r$-jets on $A$.
Therefore, dimension of fibers of $\pi'_r$ and also $\pi_r|_{S_{m,n}}$ are upper-semi-continuous,
which implies the inequality $\delta_W\leq \delta_V$.
This yields
$$s_V(m,n)-s_W(m,n)\leq \codim (V, W)$$
for every $m, n$, which yields the inequality in (ii).

For a fixed $m, n$, there exists a closed subset $F_{m,n}\subset W$ such that for every point
$\eta_V\in W$
which is not contained in $F_{m,n}$ satisfies $\delta_V=\delta_W$.
Then, for these $V$ we obtain
$$s_V(m,n)-s_W(m,n)= \codim (V, W).$$
Therefore, if $\eta_V$ is not contained in  $F=\bigcup_{m,n}F_{m,n}$,
then, we obtain the equality
$$\mldmj(\eta_V; X, \a)= \mldmj(\eta_W; X, \a)+\codim(V, W).$$

For the proof of (iii), assume that $\cha\ k =0$.
Let $f: A'\to A$ be an embedded log resolution of $(X,\a)$ with at least one exceptional divisor
 with the center $\overline{\{\eta_W\}}$.
 Then, there is an exceptional prime divisor $E\subset A'$ computing the $\mldmj(\eta_W; X,\a)$.
 If $\eta_V\in W$ is not in the union of the lower dimensional centers of the exceptional divisors of $f$,
 then, the divisor obtained by the blow-up with the center $f^{-1}(V)\cap E$ computes
 $$\mldmj(\eta_V; X, \a)=\mldmj(\eta_W; X,\a)+\codim (V, W).$$
\end{proof}

\begin{rem}
 As is seen in the proof of (iii), one can see that for a positive characteristic case if (iii) in Corollary \ref{very} does not
 hold,  it shows a counter example of the existence of resolution of singularities.

\end{rem}

\vskip.5truecm
\section{MJ-canonical  and MJ-log-canonical singularities}
\noindent
We say that $X$ has MJ-canonical (resp. MJ-log-canonical) singularities if
the pair $(X, \o_X)$ has MJ-canonical (resp. MJ-log-canonical) singularities.
We denote $\mldmj(x; X,\o_X)$ by $\mldmj(x; X)$.
In this section, we will study the nature of MJ-canonical singularities and MJ-log-canonical
singularities.

\begin{lem}[{\cite[Proposition 3.3]{ei}}]\label{emb}
  Let $x\in X$ be a closed point.
  If $X$ is MJ-canonical at $x$, then the embedding dimension satisfies $$\emb(X,x)\leq 2d-1.$$
  If $X$ is MJ-log-canonical at $x$, then the embedding dimension satisfies $$\emb(X,x)\leq 2d.$$
\end{lem}

\begin{proof} This is proved in \cite[Proposition 3.3]{ei} in case $\cha\ k=0$.
However, the main point of the proof is the formula in Corollary \ref{useful} and
we have seen in the proof of the corollary that the formula holds true for an arbitrary characteristic.
Therefore, the same proof works for the  lemma in the positive characteristic case too.
\end{proof}

\begin{defn}\label{pseudo}
We say that a variety $X$ has a pseudo rational singularity at $x\in X$ if
\begin{enumerate}
\item
  $X $ is normal around $x$;
\item
 For every partial resolution $f:Y\to X$, which means a proper birational morphism with normal $Y$,
 the equality $$f_*\omega_Y=\omega_X $$
holds around $x$;
\item
  $X$ is Cohen--Macaulay around $x$.

\end{enumerate}
\end{defn}

Note that if $\cha\ k=0$ or if $\dim X=2$, this definition is equivalent to the following:
\begin{enumerate}
\item
  $X $ is normal around $x$;
\item[(2')]
    For every resolution $f:Y\to X$ the vanishing
    $$R^if_*\o_Y=0$$
    holds for $i>0$ around $x$.
 \end{enumerate}

 The singularity $(X,x)$ satisfying (1) and (2') is called a rational singularity.

\begin{prop} \label{mjcano} Let a variety $X$ have at worst MJ-canonical singularities.
Then $X$ has normal hypersurface singularities  in codimension 2.

If $\cha\ k=0$, then $X$ is normal;  furthermore, the singularities on $X$ are
rational.
\end{prop}

\begin{proof}
  The second statement is proved in \cite[Theorem 7.7]{dd} and \cite[Corollary 3.7]{EIM}
  independently.
  Regarding the first statement, since the problem is local, we may assume that $X$ is a closed subvariety  in the
affine space $A=\AA^N$.
  Then, as $X$ is MJ-canonical, it follows that $\mldmj(\eta; X)=\mld(\eta, A, I_X^c)\geq 0,
  $
  where $I_X$ is the defining ideal of $X $ in $A$ and $c=\codim (X, A)$.

  The first statement is proved by blow-up $A$ at an irreducible component of the singular locus of $X$ and checking the discrepancy of $(A, I_X^c)$ as in the proof of Proposition \ref{mjlogcodim1}.
  But, here, we present a proof  using jet schemes discussions.

 Let $\eta\in X$ be the generic point of an irreducible closed subset of codimension one.
  Then, as $X$ is MJ-canonical, it follows that $\mldmj(\eta; X)\geq 1$.
  On the other hand, we have $\mldmj(\eta; X)\leq d-\dim\overline{\{\eta\}}=1$ by Corollary \ref{very}.
  Then, we obtain that the equality in (i) in Corollary \ref{very} holds, which yields that $X$ is
  regular at $\eta$.
  Now, $\dim \sing(X)\leq d-2$.
  Let $\zeta\in X$ be the generic point of an irreducible component of $\sing(X)$ of codimension 2.
  Then, by inequality (\ref{use2}), it follows
  $$1\leq \mldmj(\zeta, X)
\leq (m+1)d -r_m-\dim \overline{\{\zeta\}},$$
where $r_m $ is the dimension of a general fiber of $\pi_m$.
Considering the case $m=1$,
we obtain
\begin{equation}\label{normal}
1\leq   \mldmj(\zeta, X) \leq 2d-r_1-(d-2)\leq 1.
\end{equation}
Here, the last inequality is showed as follows: note that $r_1$ is the dimension of Zariski tangent space of $X$
at a general point; therefore, $r_1\geq d+1$, as the point is a singular point.
Therefore, all inequalities in (\ref{normal}) become equalities, in particular $r_1=d+1$,
which means that a general point is a hypersurface  singularity.
Hence, $X$ is a Gorenstein variety in codimension 2 and satisfies $R_1$, which yields that
$X$ is normal in codimension 2 by Serre's criteria.
\end{proof}

\begin{cor}
  A two-dimensional singularity $(X,x)$ is  MJ-canonical if and only if it is a rational double point
  in arbitrary characteristic.
\end{cor}

\begin{proof}
In case $\cha\ k=0$, this statement is proved in \cite{ei}.
The following is a characteristic free proof.
As a rational double point of dimension two is a normal hypersurface singularity, 
the canonicity in the usual sense is equivalent to  MJ-canonicity.
Therefore, a rational double point of dimension two is MJ-canonical.

Conversely, if $(X,x)$ is MJ-canonical, then, by Proposition \ref{mjcano},
it is a normal hypersurface singularity.
Therefore it is canonical in the usual sense, which yields that it is rational double.
\end{proof}

\begin{cor}
  If an MJ-canonical variety $X$ is locally a complete intersection, then $X$ is normal and
  has pseudo rational singularities.
  In particular, for $\cha\ k=0$, $X$ has rational singularities.

  \end{cor}

\begin{proof}
 As $X$ is locally a complete intersection, it is Gorenstein.
 By Proposition \ref{mjcano}, $X$ satisfies $R_1$, therefore by Serre's criteria,
 it is normal.
 Since $X$ is locally a complete intersection, we have that $ \hke-j_E=k_E$ (Remark \ref{comin})
 for every prime divisor $E$ over $X$.
 On the other hand, when we take a partial resolution $f:Y\to X$,
 for every prime divisor $E$ on $Y$, we obatain $\val_E(K_Y-f^*K_X)=k_E=\hke-j_E\geq 0$
by the assumption that $X$ is MJ-canonical.
Therefore, on $Y$, we have the inequality $K_Y\geq f^*K_X$ which yields
\begin{equation}\label{oppos}
f_*\omega_Y\supset f_*f^*\omega_X=\omega_X.
\end{equation}
Now, we obtain
  $f_*\omega_Y=\omega_X$, since the opposite inclusion of (\ref{oppos}) is trivial.

\end{proof}

\begin{defn}
Let $x\in X$ be a closed point of a variety $X$.
We say that $x$ is a normal crossing double point in $X$
if $\widehat{\o_{X,x}}=k[[x_1,\ldots, x_N]]/(x_1\cdot x_2).$

\end{defn}

\begin{prop}
\label{mjlogcodim1} Let a variety $X$ have at worst MJ-log-canonical singularities.
Then a general point of the singular locus of codimension one is  normal crossing double.
\end{prop}

\begin{proof}  Since the problem is local, we may assume that $X$ is a closed subvariety  in a smooth 
affine variety $A$.
  Then, as $X$ is MJ-log-canonical, it follows that $\mldmj(\eta; X)=\mld(\eta, A, I_X^c)\geq 0,
  $ for every (not necessarily closed) point $\eta\in X$,
  where $I_X$ is the defining ideal of $X $ in $A$ and $c=\codim (X, A)$.
  It seems to be well known that if $(A, I_X^c)$ is log-canonical, then
  $X$ is a hypersurface with at worst normal crossing double singularities in codimension 1
  when the base field is of characteristic 0.
  We will write the proof that works  for  an arbitrary characteristic case.

 Let $S\subset X$ be  an irreducible component of codimension 1 in the singular locus and
let $\eta\in S$ be the generic point.
Take $A$ as a minimal dimensional smooth ambient space around $\eta$.
Let $f:A_1\to A$ be the blow-up of $A$ with the center $S$ and let $E_1$ be the exceptional
divisor dominating $S$.
Then, we obtain
$$a(E_1;A, I_X^c)=k_{E_1}-c\cdot\val_{E_1} (I_X)+1\geq 0.$$
Here, we note that $k_{E_1}=c$. 
By the minimality of $\dim A$, every element  of $I_X$ has multiplicity $\geq 2$ at $\eta$,
which yields 
$\val_E(I_X)\geq 2$.
Then, we obtain
$$0\leq k_{E_1}-c\cdot\val_{E_1} (I_X)+1\leq c-2c+1=-c+1\leq 0.$$
Therefore all equalities should hold, in particular, $a(E_1;A, I_X^c)=0$, $c=1$ and
$\val_ {E_1}(I_X)=2$.
These mean that $X$ has hypersurface double points in codimension 1.

If a general point of $S$ is  not a normal crossing double point, then $E_1$ and the proper transform
$X_1$ contact with the order $\geq 2$ along a closed subset $S_1$ dominating $S$.
Let $A_2\to A_1$ be the blow-up with the center $S_1$ and $E_2$ the exceptional divisor
dominating $S_1$.
Let  $E_1'\subset A_2$ and $X_2\subset A_2$ be the proper transforms  of $E_1$ and $X_1$, respectively.

Then, three divisors $E_2$,  $E_1'$ and  $X_2$ in $A_2$ 
still have an intersection at a closed subset $S_2$ of dimension $d-1$.
Now, blow up $A_3\to A_2$ with the center $S_2$ and let $E_3$ be the exceptional
divisor dominating $S_2$.
Then, the log discrepancy at $E_3$ is
$$a({E_3}; A, I_X)=k_{E_3}-\val_{E_3}I_X+1\leq 3-5+1=-1,$$
a contradiction.
Therefore, $E_1$ and $X_1$ have the reduced intersection at $S_1$ over general points of $S$,
which implies that $X$ is normal crossing double at a general point of $S$.
\end{proof}

\begin{thm}\label{isiireguera} Let $k$ be an algebraically closed field.
Let $p\in X$ be a closed point of an arbitrary variety $X$ over $k$.
 A pair $(X, B)$ consisting of   $X$
  and an effective $\RR$-Cartier divisor $B$ on $X$ satisfies
 \begin{equation}\label{futousiki}
 \dim X-1\leq {\mldmj}(p;X,{}B)
 \end{equation}
if and only if
either
\begin{enumerate}
\item[(i)]   $\dim X\geq 2$, $B=0$ and $(X,p) $ is a compound Du Val singularity,
\item[(ii)]  $B=0$, and $(X,p)\subset (\AA^3, 0)$  is given by:
$$
\begin{array}{ll}
& xy=0, \\
& z^2+x y^2=0 \text{ if } \cha\ k \neq 2, \text{ or } \\
& z^2+x y^2 + yz g(x,y)=0 \text{ if } \cha\ k = 2, \text{ where } mult g \geq 1 \text{ and either } g=0 \text{ or
} g(x,0) \neq 0.
\end{array}
$$
\item[(iii)] $(X,p)$ is non-sigular and $0\leq \mult_pB\leq 1$.
\end{enumerate}

\noindent
In  cases $(\mathrm i)$ and $(\mathrm {ii})$, we have ${\mldmj}(p;X,{{\mathcal J}_X})= \dim X -1$ and
 in  case $(\mathrm {iii})$ we have
${\mldmj}(p;X,{}B)=\mld(p; X, B)= \dim X -\mult_pB$ and the minimal  log discrepancy is computed by the exceptional divisor of the first blow-up at $p$.
\end{thm}

\begin{proof} As in \cite[Theorem 4.1]{ir}, it suffices to show that equality in (\ref{futousiki}) holds if and only if (i) or (ii) are satisfied. Germs of varieties $(X,p)$, where $p$ is a closed point, satisfying $\mldmj(p,X)= \dim X -1$ are called top singularities in \cite{ir}. \\

Recall the invariant $\tau(X, p)$ introduced in
\cite{Hi1}. In the case of a germ of an hypersurface $(X,p)$ in $(\AA^{d+1}_k, 0)$ defined by $f \in k[x_1, \ldots, x_{d+1}]$,
$\tau$ is the smallest possible dimension of a linear subspace $V_0$
of $V=kx_1+kx_2+\cdots+kx_{d+1}$ such that the initial term $\mbox{in} f$ of $f$ lies in the subalgebra
$k[V_0]$ of $k[x_1,\ldots, x_{d+1}]$. It is known to be an invariant of $(X,p)$.  \\

The following {\it results}, proved in \cite{ir} if $\text{char } k=0$, are based on equality (\ref{basicequation}), hence, remain true if   $\text{char } k>0$:
\begin{enumerate}
\item[(R0)] \cite[Lemma 3.5]{ir} Let $X$ be a $d$-dimensional variety and let $X' \subset X$ be a
$(d-c)$-dimensional subvariety, which is defined as the zero locus
of $c$ elements of ${\cal O}_X$. Let $p$ be a closed point in
$X'$. If $(X', p)$ is a top singularity,  then $(X, p)$ is a
top singularity. \\
\item[(R1)] \cite[Lemma 3.6]{ir} If $X$ has a top singularity at $p$, then $X$ has a
hypersurface singularity of multiplicity $2$ at $p$.
\item[(R2)] \cite[Lemma 3.20]{ir} Let $(X, p)\subset (\AA^{d+1}_k, 0)$ be a germ of a hypersurface singularity  
 of multiplicity $2$ at $p$  and $\tau(X, p)
=1$. Hence, the equation of $(X, p)\subset (\AA^{d+1}, 0)$
is expressed as 
$$f=x_{1}^2 -g(x_2, \ldots, x_{d+1}) + x_1g'(x_1,\ldots, x_{d+1})=0.$$
Here, we note that
 $x_1g'=x_{1}^2 -g(x_2, \ldots, x_{d+1}) \in x_{1} (x_1, \ldots, x_{d+1})^2$. Then, if $(X, p)$ is a top singularity, we have $d \geq
2$ and $\text{mult } g =3$.\\
\item[(R3)] \cite[Proposition 3.23]{ir}  In the conditions of (R2), suppose also that the initial form of $g$ has only one factor.  Hence, the equation of $(X, p)\subset (\AA^{d+1},0)$
gives
   $$
   \begin{array}{l}
x_1^2 \ + \ x_2^3 + g_3(x_3, \ldots,
x_{d+1}) \ x_2 \ + \ g_4(x_3, \ldots, x_{d+1})
\  \\
\ \ \ \ \ \ \ \ \ \ \in x_1 (x_3, \ldots , x_{d+1})^2 + x_1 x_2 (x_3, \ldots , x_{d+1}) + (x_1x_2^2) + x_2^2 (x_3, \ldots , x_{d+1})^2
\end{array}
$$
where $\text{mult} \ g_i
\geq i$, for $i=3,4$.  Then, if $(X, p)$ is a top singularity, we have either $\text{mult} \ g_3=3$ or $4 \leq
\text{mult} \ g_4 \leq 5$. \\

\end{enumerate}

On the other hand, in  pages 256--268 of \cite{Li}, a characterization of pseudo rational double points over any field is given.

Restricting to an algebraically closed field k, the following is proved:
Let us assume the {\it hypothesis} H1:
\begin{enumerate}
\item[H1)] $X \subset \AA_k^3$ be a surface of multiplicity $2$.
Then $\tau \leq 3$ and we
have:
\end{enumerate}
\begin{enumerate}

\item[
{\bf Case I}]: $\tau=3$ if and only if $(X,p)\subset (\AA^3, 0)$ has $A_1$-singularity.

\item[
{\bf Case II}] (II a) in \cite{Li}): $\tau=2$ is equivalent to $(X,p)\subset (\AA^3, 0)$ having $A_n$-singularity $(n \geq 2)$ if we assume that $(X,p)$ is pseudo rational.

\end{enumerate}
If $\tau=1$,  the equation of $X$ gives $z^2-G(x,y) \in z M^2$ with $\text{mult } G \geq 3$, where $M$ is the maximal ideal at the origin.
Besides, the hypothesis $X$ pseudo rational implies that

\begin{enumerate}
\item[
H2)] $\text{mult } G=3$.
\end{enumerate}
Let $\overline G$ be the initial form of $G$. Then one of the following cases
occurs:

\begin{enumerate}
\item[
{\bf Case III}] (III c in \cite{Li}): $\tau=1$ and $\overline G$ has $3$ factors. This is equivalent to $X$ being a $D_4$-singularity

\item[{\bf Case IV}]: $\tau=1$ and  $\overline G$ is the product of a linear factor and the square of another factor. This is equivalent to $(X, p)\subset (\AA^3, 0)$ having a $D_n$-singularity $(n \geq 5)$ if we assume that $(X, p)$ is pseudo rational.

\item[{\bf Case V}]: $\tau=1$ and  $\overline G$ has only $1$ factor. Then, either $(X, p)$ is an $E_6$-singularity or it can be expressed as
  $$z^2 + y^3 + \rho x^3 y + \sigma x^5 \in (zxy, zy^2, x^2 y^2, x^3z).$$
\end{enumerate}
Again $(X, p)$ pseudo rational implies:
\begin{enumerate}

\item[
H3)] either $\rho$ is a unit or $\sigma$ is a unit.

\end{enumerate}

Finally, we have 

\begin{enumerate}
\item[
{\bf Case V c}]: $\rho$ being a unit is equivalent to $(X, p)$ being an $E_7$-singularity.
\item[
{\bf Case V d}]: $\rho$ not a unit and $\sigma$ a unit, is equivalent to $(X, p)$ being an $E_8$-singularity.\\

\end{enumerate}

Note that rational double points are top singularities (apply the proof for $\text{char } k =0$ given in \cite{ir}, example 3.12).
From this, applying (R0), it follows that (i) in the theorem implies equality in (16). The fact that (ii) implies equality in (16) can also be checked. In fact, for the last case in (ii), note that the surface given by $z^2+x y^2 + yz g(x,y)=0$  ($\text{char } k = 2$), where $\mult g\geq 1$ and either $g=0$ or $g(x,0)=0$,
can be desingularized by the blow-up of $(y,z)$; if $E$ is the exceptional curve appearing then $\widehat{k}_E=1$ and $\text{ord}_E(\j_X)=1$.\\

Finally, let us prove that if $(X, p)$ is a top singularity, then either (i) or (ii) hold. In fact, let $(X', p)$ be the surface obtained by a general cut of $(X, p)$ with $d-2$ hyperplanes. Since (R1) holds for $(X, p)$, we have that (H1) holds for $(X', p)$. Analogously, since (R2) holds for  $(X, p)$, it follows that for $(X', p)$ we have that  if $\tau=1$ then (H2) holds. Finally, (R3) for $(X, p)$ implies that if $(X', p)$ satisfies the hypothesis of Case V,  either  $(X', p)$ is an $E_6$-singularity or (H3) holds.

We thus  conclude that if we assume that $(X', p)$ is pseudo rational in Cases II and  IV, then $(X, p)$ being a top singularity would imply that $(X', p)$ is a rational double point; hence, $(X, p)$ is a compound Du Val singularity.

Finally, for cases II and IV without the hypothesis of pseudo rational, we obtain
\begin{enumerate}
\item[1.] If $(X', p)$ is in case II and not pseudo rational, then it can be expressed as $xy=0$.

\item[2.] If $(X', p)$ is in case IV and not pseudo rational, then it can be expressed as:
$$
\begin{array}{ll}
& z^2+x y^2=0 \text{ if } \cha\ k \neq 2, \text{ or } \\
& z^2+x y^2 + yz g(x,y)=0 \text{ if } \cha\ k = 2, \\
&\text{ where } \mult g \geq 1 \text{ and either } g=0 \text{ or
} g(x,0) \neq 0.
\end{array}
$$
\end{enumerate}
Thus, we conclude the result.

\end{proof}

\begin{rem}\label{bound}
For a closed point $x\in X$ in a variety $X$.
Define $$s_m(x):=(m+1)d-\dim \pi_m^{-1}(x).$$
We know that $\mldmj(x; X)=\inf_{m\in \NN} s_m(x)$ by the formula
 (\ref{basicequation}) in
Corollary \ref{useful}.
  The proof of the previous theorem shows that  $s_m\geq d-1$ ($m\leq 5$) yields that $\mldmj(x; X)\geq
  d-1$.
\end{rem}

\begin{prop}\label{cdv}
Let $X$ be a  variety over an algebraically closed field $k$.
Assume that $X$ has at worst MJ-canonical singularities.
Then,
$X$ has at worst cDV singularities in codimension 2.
\end{prop}

\begin{proof}
In Proposition \ref{mjcano}, it  is proved that $X$ has normal in codimension 2.
Let $\eta\in X$ be the generic point of an irreducible component of the
singular locus of codimension 2.
As $X$ is MJ-canonical, it follows $\mldmj(\eta; X)\geq 1$; therefore, by the
Corollary \ref{useful}
$$
1\leq\mldmj(\eta; X)
= \inf_m\left\{(m+1)d-(\dim  \overline{\{\eta\}}+r_m)\right\},$$
where $r_m=\dim \pi_m^{-1}(x)$ ($m\in \NN$) for a general closed point
$x\in \overline{\{\eta\}}$.
Considering the cases for $m=1,\ldots, 5$, we obtain  
$$s_m(x):=(m+1)d-r_m\geq d-1$$
 for a general point
$x\in \overline{\{\eta\}}$.
Therefore, by Remark \ref{bound},  we obtain that $\mldmj(x;X)=d-1$ at a general point $x
\in \overline{\{\eta\}}$.

\end{proof}

\begin{cor}\label{hyper}
Let $X$ be an MJ-canonical  quasi-projective variety of dimension 3 over an algebraically
closed field $k$.
Then,
a general hyperplane section $H$ of $X$ has at worst Du Val singularities.
\end{cor}

\begin{proof} Let $X$ be a locally closed subvariety of $\PP^N$.
As the linear system of hyperplane in $\PP^N$ is very ample,
we can apply the original Bertini's theorem (see, for example, \cite[Theorem 8.18, II]{ha})
that works for arbitrary characteristic.
Then, it follows that a general hyperplane section $H$ is
  non-singular away from the singular locus of $X$.
  Therefore, if $\dim \sing X=0$, then $H$ is non-singular.
  So we assume that $\dim \sing X=1$.
  By Proposition \ref{cdv}, a general hyperplane section intersects the singular locus
  at finite number of cDV points.
  We have only to show that the intersection is general at each point $x\in H\cap \sing X$.
 Let $|\o_{\PP^N}(1)|$ be the complete linear system of hyperplanes on $\PP^N$.
 For a point $x\in \sing X$, we define a subset $D_x\subset |\o_{\PP^N}(1)|$
 as
 $$D_x=\{\H\in |\o_{\PP^N}(1)|  \mid X\subset \H \mbox{\ or\ } (\H\cap X, x)\mbox{\ not\  rational\
 double}\}.
$$
As $\o_{\PP^N}(1)$ is very ample, the canonical $k$-linear map
$$\varphi_x: \Gamma(\PP^N,\o_{\PP^N}(1)) \to \o_{\PP^N}(1)\otimes\o_X/{\frak{m}}_x^2\simeq
\o_X/{\frak{m}}_x^2$$
is surjective.
Let $\widetilde D_x\subset \Gamma(\PP^N,\o_{\PP^N}(1))$ be the subset corresponding to $D_x$ for a cDV point $x\in \sing X$,
then $\widetilde D_x$ is the pull-back by $\varphi_x$ of the proper closed subset of $k^4={\frak m}_x/{\frak m}_x^2\subset \o_X/{\frak{m}}_x^2$.
Then, $$\dim \widetilde D_x\leq N+1-5+3=N-1;$$
 therefore $$\dim D_x\leq N-2.$$
On the other hand, if $x\in \sing X$ is not a cDV point, then $$\dim D_x\leq N-1.$$
Let $D\subset \sing X\times |\o_{\PP^N}(1)|$ be the set $\{\langle x, \H\rangle \mid \H\in D_x\}$.
Then, as $\sing X$ is of 1-dimensional and non cDV singularities are isolated,
 we have $$\dim D\leq N-1.$$
Hence, the image $p(D)\subset |\o_{\PP^N}(1)| $ of $D$ by the projection
$$p:\sing X\times |\o_{\PP^N}(1)|\to |\o_{\PP^N}(1)|$$ has dimension $N-1\leq N=\dim |\o_{\PP^N}(1)|$, which yields that general elements of $|\o_{\PP^N}(1)|$ is not in $p(D)$.
This completes the proof.
\end{proof}

The usual canonical version of this statement is proved in \cite{hiro} under certain conditions.

\vskip.5truecm
Here, we list open problems for the positive characteristic case, which are all proved for characteristic 0 (see \cite{EIM} for (1) and see \cite{ei} for (2),(3),(4)):

\vskip.5truecm
{\bf Open problems for positive characteristic case:}
\begin{enumerate}
\item Is an MJ-canonical singularity normal? Cohen--Macaulay?
\item Is the map $X\to \ZZ$, $x\mapsto \mldmj(x, X)$ lower semi-continuous?
\item Is an MJ-canonical ( MJ-log-canonical) singularity open condition?
\item Is a small deformation of MJ-canonical singularity again MJ-canonical?
\end{enumerate}

Here, we note that if there exist resolutions of singularities, then we would have the
affirmative answer to (2), (3), and (4).
On the other hand, without resolutions we can prove these if the following natural conjecture
holds (this will be discussed in a forthcoming paper by one of the authors):
\begin{conj}
  Let $0\leq\delta\leq d$, there is a number $N_{\delta,d} $  depending only on $\delta$ and $d$
  such that if
  $$s_m(x) \geq \delta\geq 0, \ \mbox {for\ all\ }m\leq N_{\delta,d}$$
  then $\mldmj(x, X)\geq \delta.$

\end{conj}

We observe that this holds true for $\delta=d-1$.
Indeed, we can take $N_{d-1,d}=5$ as is seen in the proof of Remark \ref{bound}.

This conjecture is equivalent to the following conjecture (this will be showed also  in the
forthcoming paper):
\begin{conj}
There exists a number $N_d$ depending only on $d$ such that
$$\min \{s_m(x)\mid m\leq N_d\}=\mldmj(x;X), \ \mbox{if}\  \mldmj(x;X)\geq 0,$$
 and 
  $$s_m(x) < 0, \ \mbox {for\ some\ }m\leq N_d , \ \mbox{if}\  \mldmj(x;X)=-\infty.$$

\end{conj}

\begin{rem} We stated the results in this section under the condition that
$k$ is algebraically closed.
However,  we can weaken this condition such that $k$ is perfect
in all results except for Theorem \ref{isiireguera} and  Corollary \ref{hyper}.

\end{rem}

\noindent
\vskip.5truecm
\noindent
 \makeatletter \renewcommand{\@biblabel}[1]{\hfill#1.}\makeatother

\vskip1truecm

\noindent Shihoko Ishii, \\ Graduate School of Mathematical Science,
University of Tokyo, \\
3-8-1  Komaba, Meguro, 153-8914 Tokyo, Japan. \\

Current address: Department of Mathematics, Tokyo Woman's Christian University,\\
2-6-1 Zenpukuji, Suginami, 167-8585 Tokyo, Japan.\\

\noindent Ana J. Reguera, \\ Dep. de \'Algebra, Geometr\'ia y Topolog\'ia,  Universidad de Valladolid,\\
Paseo Bel\'en 7,
47011 Valladolid, Spain. \\
E-mail: areguera@agt.uva.es

\end{document}